\documentclass[11pt]{amsart}
\usepackage{amsmath,amssymb,amsfonts,amsthm, amscd,indentfirst}
\usepackage{amsmath,latexsym,amssymb,amsmath,
	amscd,amsthm,amsxtra}
\usepackage{hyperref}\usepackage{url}
\usepackage{color}
\usepackage{pdfpages}
\usepackage{graphics}
\usepackage{enumitem}

\usepackage{graphicx}
\usepackage{caption}

\usepackage{epsfig,here}
\usepackage{subfigure,here}
\numberwithin{equation}{section}

\newtheorem{Theorem}{Theorem}[section]
\newtheorem{Lemma}{Lemma}[section]
\newtheorem{Proposition}{Proposition}[section]
\newtheorem{Corollary}{Corollary}[section]

\newtheorem{Remark}{\textbf{Remark}}[section]

\def\S{\Sigma}
\def\n{\nabla}

\def\p{\partial}

\def\a{\alpha}
\def\b{\beta}
\def\n{\nabla}

\def\p{\partial}

\def\a{\alpha}
\def\b{\beta}
\def\g{\gamma}
\def\d{\delta}

\def\n{\nabla}
\def\<{\langle}
\def\>{\rangle}

\def\n{\nabla}

\def\p{\partial}

\def\a{\alpha}
\def\b{\beta}
\def\g{\gamma}
\def\d{\delta}

\def\rr{\mathbb{R}}
\def\ve{\varepsilon}

\begin{document}
	
	\title[New monotonicity for $p$-capacitary functions]{New monotonicity for $p$-capacitary functions in $3$-manifolds with nonnegative scalar curvature}
	\author{Chao Xia}
	\address{School of Mathematical Sciences\\
		Xiamen University\\
		361005, Xiamen, P.R. China}
	\email{chaoxia@xmu.edu.cn}
	\author{Jiabin Yin}
	\address{School of Mathematics and Statistics\\ Guangxi Normal University\\  541004, Guilin, P.R. China}
	\email{jiabinyin@126.com}
	\author{Xingjian Zhou}
	\address{School of Mathematical Sciences\\
		Xiamen University\\
		361005, Xiamen, P.R. China}
	\email{zhouxingjian@stu.xmu.edu.cn}
	\thanks{CX is supported by the NSFC (Grant No. 12271449). JY is supported by the NSFC (Grant No. 12201138) and Mathematics Tianyuan fund project (Grant No. 12226350).}
	
	\begin{abstract}
		In this paper, we derive general monotone quantities and geometric inequalities associated with $p$-capacitary functions in asymptotically flat $3$-manifolds with simple topology and nonnegative scalar curvature. The inequalities become equalities on the spatial Schwarzschild manifolds outside rotationally symmetric spheres. This generalizes Miao's result \cite{M1}  from $p=2$ to $p\in (1, 3)$.  As applications, we recover mass-to-$p$-capacity and $p$-capacity-to-area inequalities due to Bray-Miao \cite{BM} and Xiao \cite{Xiao}.
		
	\end{abstract}

	\date{}
	\keywords{}

	\maketitle

	\section{Introduction}
	For an asymptotically flat $3$-manifold $(M, g)$ with outmost minimal boundary and non-negative scalar curvature, the Riemannian-Penrose inequality says that its ADM mass \begin{eqnarray}\label{RPI}
		\mathfrak{m}_{ADM}(M)\ge \sqrt{\frac{|\p M|}{16\pi}},
	\end{eqnarray}
	with equality holding if and only if $(M, g)$ is a spatial Schwarzschild manifold. It was first proved in the case of connected boundary by Huisken-Ilmanen \cite{HI} via establishing a Geroch-type monotonicity formula for a level set formulation of inverse mean curvature flow. The general case and the higher dimensional inequality in dimensions less than eight was proved by Bray \cite{Br} and  Bray-Lee \cite{BL} via a conformal flow of metrics. Recently, Agostiniani-Mantegazza-Mazzieri-Oronzio \cite{AMMO} gave a new proof of the 3-dimensional Riemannian-Penrose inequality in the case of connected boundary via establishing monotonicity formula for $p$-harmonic functions. Before that, Bray-Kazaras-Khuri-Stern \cite{BKKS} and Agostiniani-Mazzieri-Oronzio \cite{AMO} also gave two new  proofs of the Riemannian positive mass theorem via establishing various monotonicity formulas for harmonic functions.  

	On the other hand, Bray \cite{Br} applied the Riemannian positive mass theorem to prove a mass-capacity inequality $$\mathfrak{m}_{ADM}\ge {\rm Cap}(\p M)$$ for  asymptotically flat $3$-manifolds with minimal boundary, where ${\rm Cap}(\p M)$ is the capacity of $\p M$ in $M$. Later, Bray-Miao \cite{M1} extended Bray's mass-capacity inequality to manifolds with nonnegative Hawking mass and connected boundary that is not necessary minimal,
	\begin{eqnarray}\label{mass-to-capacity}
		\frac{\mathfrak{m}_{ADM}(M)}{{\rm Cap}(\p M)}\ge 1-\left(\frac{1}{16\pi}\int_{\p M}H^2\right)^{\frac12},
	\end{eqnarray}
	and Xiao \cite{Xiao} extended Bray-Miao's result to the $p$-capacity.
	
	Recently,  inequality \eqref{mass-to-capacity} was shown by Miao \cite{M1} without the assumption of nonnegative Hawking mass, which offers another proof of 3-dim PMT, see Miao \cite{M2}. To prove that, Miao \cite{M1} found several monotone quantities associated to capacitary functions on $3$-manifolds with simple topology and non-negative scalar curvature, by which he established the following geometric inequalities
	\begin{eqnarray}
		&&4\pi+\int_{\p M} H|\n u|\ge 3\int_{\p M} |\n u|^2,\label{miao-ineq1}\\
		&&4\pi-\int_{\p M}  |\n u|^2\le 4\pi \frac{\mathfrak{m}_{ADM}(M)}{{\rm Cap}(\p M)},\label{miao-ineq2}
	\end{eqnarray}
	where $u$ is the capacitary function on $M$ and $H$ is the mean curvature. Moreover, equality in both \eqref{miao-ineq1} and \eqref{miao-ineq2} holds if and only if  $(M, g)$ is isometric to $\rr^3$ minus a round ball, see \cite[Theorems 3.1 and 3.2]{M1}.  Inspired by Bray's work \cite{Br}, by using  a conformal metric which preserves the nonnegativity of scalar curvature and the harmonicity of a function, Miao \cite{M1} promoted the above geometric inequalities to the following ones: for $l>0$,
	\begin{eqnarray}
		&&4\pi+l\int_{\p M} H|\n u|\ge l(4-l)\int_{\p M} |\n u|^2,\label{miao-ineq3}\\
		&&4\pi\left(2-\frac{1}{l}\right)-l\int_{\p M}  |\n u|^2\le 4\pi \frac{\mathfrak{m}_{ADM}(M)}{{\rm Cap}(\p M)},\label{miao-ineq4}
	\end{eqnarray}
	which become equalities in a spatial Schwarzschild manifold outside a rotationally symmetric sphere, see \cite[(7.10), (7.27)]{M1}. 
	Similar monotone quantities and geometric inequalities on capacitary functions have been also found by Oronzio \cite{O} independently. We also remark that various monotone quantities  for Green's functions and $p$-harmonic Green's
	functions on $3$-manifolds with non-negative scalar curvature and their applications were studied by Agostiniani-Mazzieri-Oronzio \cite{AMO},  Munteanu-Wang \cite{MW} and Chan-Chu-Lee-Tsang \cite{CCLT} respectively.  {These works were inspired and preceded by their harmonic counterparts \cite{AFM20,CM,CM14}, starting with Colding's breakthrough \cite{C12}.}
	
	Continuing on Miao's work \cite{M1}, Hirsch-Miao-Tam \cite{HMT} also found several monotone quantities associated to $p$-capacitary functions on $3$-manifolds with simple topology and non-negative scalar curvature, by which they established geometric inequalities  associated to $p$-capacitary functions in the same spirit of \eqref{miao-ineq1} and \eqref{miao-ineq2}, whose model space is $\rr^3$ minus a round ball. Before that, Agostiniani-Mantegazza-Mazzieri-Oronzio \cite{AMMO} also derives a similar monotone quantity associated to $p$-capacitary functions, aiming at Riemannian-Penrose inequality as we mentioned at the beginning.
	However, there seems no corresponding conformal metric which preserves the $p$-harmonicity and it is not known whether Hirsch-Miao-Tam's monotone quantities can be promoted to the ones which become equality in a spatial Schwarzschild manifold. 
	
	In this paper, we give an affirmative answer to this question, that is, we find general monotone quantities associated to $p$-capacitary functions whose model space is a spatial Schwarzschild manifold outside a rotationally symmetric sphere. In order to state  our main result, we first introduce several preliminaries.	
	
	A $3$-dimensional Riemannian manifold $(M, g)$ is said to be a one-end asymptotically flat  if there is a compact subset $K$ of $M$ and a diffeomorphism   $\Phi : M \setminus K \rightarrow \mathbb{R}^{3} \setminus  \overline{\Omega}$, where $\overline{\Omega}$ is a compact subset of $\mathbb{R}^{3}$ and with respect to the standard coordinates on $\mathbb{R}^{3}$, $g$ satisfies
	$$g_{ij} = \delta_{ij} + \sigma_{ij},$$
	where $\sigma$ is a symmetric $(0,2)-$tensor such that $\sigma=O_2\left( |x|^{- \tau} \right)$ for some $\tau > \frac{1}{2}$, which means that
	\begin{equation}\label{af}
		\partial_{J} \sigma = O \left( |x|^{- \tau - |J| } \right) \ \ \text{as} \ x \rightarrow \infty \  \text{for every multi-index} \ J \ \text{with} \ |J| \le 2.
	\end{equation}	
	
	The  ADM mass of $(M, g)$, which has been introduced by the physicists Arnowitt, Deser and Misner in \cite{ADM},  is defined by
	\begin{equation}
		\mathfrak{m}_{ADM} : = \frac{1}{16\pi}\lim\limits_{r \rightarrow \infty} \int_{S_r} \left( \partial_{j} g_{ij}  - \partial_{i} g_{jj} \right) \frac{x^i}{|x|}dA_{\bar g},
	\end{equation}
	where $S_r=\{|x|=r\}$ and $dA_{\bar g}$ is the volume form induced from the Euclidean metric.
	The scalar curvature $R$ of $(M, g)$ is assumed to be integrable so that the ADM mass of $(M, g)$ exists and its value does not depend on the asymptotically flat coordinate chart, see \cite{B}.
	
	Given $m>0$, the spatial Schwarzschild manifold of ADM mass $m$ is the $3$-manifold with boundary given by $(\mathcal{M}_{m}^{3}, g_{m})$, where \begin{eqnarray}\label{schwarzschild}
		\mathcal{M}_{m}^{3} := \mathbb{R}^{3} \setminus {B_{\frac{m}{2}}(0)},\quad  g_{m} := \left(1 + \frac{m}{2 |x|}\right)^{4} \bar g,
	\end{eqnarray}
	where $\bar g$ is the Euclidean metric.
	Given $r_0\ge \frac{m}{2}$, we use $(\mathcal{M}_{m, r_0}^{3},  g_{m})$ to denote the spatial Schwarzschild manifold of mass $m$ outside a rotationally symmetric sphere,
	\begin{eqnarray}\label{schwarzschild1}
		\mathcal{M}_{m, r_0}^{3} := \mathbb{R}^{3} \setminus {B_{r_{0}}(0)},\quad  g_{m} = \left(1 + \frac{m}{2 |x|}\right)^{4} \bar g. 
	\end{eqnarray}
	In the case $r_0=\frac{m}{2}$,  it reduces to the spatial Schwarzschild manifold with horizon boundary, namely, $\p \mathcal{M}_{m, \frac{m}{2}}=\p \mathcal{M}_{m}$ is a outmost minimal surface.	
	
	Let $(M, g)$ be a one-end asymptotically flat $3$-manifold with boundary $\p M=\S$. For $p\in (1, 3)$, let $u\in W^{1,p}_{loc}(M)$ be the weak solution to
	\begin{equation}\label{p-laplace}
		\left\{
		\begin{aligned}
			\Delta_{p}u:={\rm div}(|\nabla u|^{p-2}\nabla u) &= 0\ \ {\rm in}\ \ M\\
			u &= 0\ \ {\rm on} \ \ \partial M\\
			u(x) &\rightarrow 1\ \ {\rm as}\ \ |x| \rightarrow \infty . 
		\end{aligned}\right.
	\end{equation}
	The $p-$capacity of $\Sigma$ in $(M, g)$ is defined by
	\begin{eqnarray*}
		\operatorname{Cap}_p(\p M) := \inf \left\{\int_M|\nabla \phi|^p\right\},
	\end{eqnarray*}
	where the infimum is taken over all Lipschitz functions $\phi$ with compact support such that $\phi=1$ at $\Sigma$.
	It is clear that $\operatorname{Cap}_p(\p M)$ is related to the solution $u$ by
	\begin{eqnarray*}
		\operatorname{Cap}_p(\p M)=\int_M|\nabla u|^p=\int_{\{u=t\}}|\nabla u|^{p-1},
	\end{eqnarray*}
	when $\{u=t\}$ is a regular level set of $u$.
	We denote
	\begin{eqnarray*}
		a :=\frac{3-p}{p-1}, \quad \hbox{ and }\quad \mathfrak{c}_{p} : = \left( \frac{\operatorname{Cap}_{p}}{4 \pi} \right)^{\frac{1}{p-1}}.
	\end{eqnarray*}
	It is known that, $u$ has an asymptotic expansion (See \cite{KV},  \cite[Theorem 4.1]{C})
	\begin{equation} \label{equ:asymptotic u}
		u=1 - \frac{\mathfrak{c}_{p}}{a} r^{-a} + O_2\left( r^{- a - \tilde{\tau}} \right), \quad \text{as} \ r=|x|\rightarrow \infty, 
	\end{equation}
	for any $0 < \tilde{\tau} < \min \{ \tau,1 \} $.
	
	Our main result is as following.
	\begin{Theorem}\label{Thm:1.02}
		Let $(M,g)$ be a  $3$-dimensional, complete, one-end asymptotically flat manifold with boundary $\p M=\Sigma$. Assume that $\S$ is connected and $H_2(M,\Sigma) = 0$ and $(M, g)$ has nonnegative scalar curvature. Let $p\in (1, 3)$ and $u$ be the weak solution to \eqref{p-laplace}.
		For any $k\in(-1,1)$,
		denote 
		$$
		m:=2 \operatorname{sgn}(k) \left(I_a(k)\mathfrak{c}_{p}\right)^{\frac{1}{a}}, \hbox{ and } r_{0}:=\frac{m}{2k}=|k|^{-1}\left(I_a(k)\mathfrak{c}_{p}\right)^{\frac{1}{a}},
		$$ 
		where $a:=\frac{3-p}{p-1}$, $\operatorname{sgn}(k)$ is the sign function on $k$ and 
		$$I_{a}(k):= \int_{0}^{|k|} s^{a-1}(1+\operatorname{sgn}(k)s)^{-2a} ds.$$
		Then the following inequalities hold:
		\begin{equation}\label{geom-ineq-10}
			\begin{aligned}
				& 4 \pi - \frac{(1+k)^{2}}{(1-k)^{2}} (\eta(r_{0}))^{2} \int_{\Sigma} |\nabla u|^2
				\\ \ge & \frac{(1+k)^{2}r_{0}}{\eta(r_{0})} \frac{(1-a\eta(r_{0}))}{m} \left\{
				4 \pi \frac{(1-k)^{2}}{(1+k)^{2}} - \int_{\Sigma} (\frac{H}{2})^{2} + \int_{\Sigma} \left( \frac{H}{2}-\eta(r_{0})|\nabla u| \right)^{2}
				\right\},
			\end{aligned}
		\end{equation}
		\begin{equation}\label{geom-ineq-20}
			\begin{aligned}
				& 4 \pi -\frac{(1+k)^{2}}{(1-k)^{2}} (\eta(r_{0}))^{2}  \int_{\Sigma} |\nabla u|^2  \le 8\pi a\left( \mathfrak{m}_{ADM}-m \right) \frac{(1-a\eta(r_{0}))}{m} .
			\end{aligned}
		\end{equation}
		Here $\eta(t)$ is a function defined by \eqref{equ:eta} below and $$\eta(r_{0})=\mathfrak{c}_{p}^{-1}r_{0}^{a}\left(1+k\right)^{2a-1}\left(1-k\right). $$
		Moreover, equality in each of the above inequalities holds for some $k$ if and only if $(M, g)$ is isometric to either the spatial Schwarzschild manifold of mass $m$ outside a rotationally symmetric ball, $(\mathcal{M}_{m, r_0}^{3},  g_{m})$ with $m=2r_0k>0$ or the Euclidean space outside a rotationally symmetric ball $(\mathbb{R}^3\setminus B_{r_0}(0), \delta)$.
		
	\end{Theorem}
	\begin{Remark}
		\item[(i)] When $k=0$,  the quantities are understood in the limit sense, that is, $$I_{a}(0)=0, \quad m=0, \quad r_{0}=\left(\frac{\mathfrak{c}_{p}}{a}\right)^{\frac{1}{a}}, \quad \eta(r_{0})=\frac{1}{a}, \quad \frac{1-a\eta(r_{0})}{m}=\lim_{k\to 0}\frac{1-a\eta(r_{0})}{m}=\frac{a^2}{a+1}\left(\frac{\mathfrak{c}_{p}}{a}\right)^{-\frac{1}{a}}.$$ 
		\item[(ii)] When $-1<k<0$, equality holds for $(\mathcal{M}_{m, r_0}^{3},  g_{m})$ with $m<0$. This is the so-called Schwarzschild ZAS metric, which has cone-type singularities. See Bray-Jauregui \cite{BJ} and Miao \cite{M2} for more discussion on this metric. Hence, for our theorem, equality cannot occur for  $-1<k<0$. 	\end{Remark}
	
	When $k=1$, we may regard $r_0=\frac{m}{2}$  and $\frac{\eta(r_0)}{1-k}=I_{a}(1)2^{2a-1}$ in Theorem \ref{Thm:1.02} so that we have the following.
	\begin{Theorem}\label{Thm:1.01}
		Let $(M,g)$ and $u$ be as in Theorem \ref{Thm:1.02}. Denote $m=2 \left(I_a(1)\mathfrak{c}_{p}\right)^{\frac{1}{a}}.$
		Then we have
		\begin{equation}\label{geom-ineq-1}
			\begin{aligned}
				&4 \pi
				+ 2 \int_{\Sigma} H |\nabla u|
				- 2^{4a} (I_{a}(1))^{2} \int_{\Sigma} |\nabla u|^2
				\ge 0, \end{aligned}
		\end{equation}
		\begin{equation}\label{geom-ineq-2}
			\begin{aligned}&4 \pi (1+2a)
				- 2^{4a} (I_{a}(1))^{2} \int_{\Sigma} |\nabla u|^2
				\le 8\pi a \frac{\mathfrak{m}_{ADM}}{m}. 		\end{aligned}	\end{equation}
			Moreover, equality in each of the above inequalities holds if and only if $(M, g)$ is isometric to the  spatial Schwarzschild manifold  of mass $m$, $(\mathcal{M}_{m}^{3},  g_{m})$.
		\end{Theorem}
		
		\begin{Remark}
			\item[(i)] By a simple calculation, we see that when $p=2$, inequalities \eqref{geom-ineq-10} and \eqref{geom-ineq-20} reduce to \eqref{miao-ineq3} and \eqref{miao-ineq4} (with $k-1=l$), which was proved by Miao via a conformal promotion of corresponding inequalities modelled on the Euclidean space outside a ball. 
			\item[(ii)] The integral $I_{a}(1)$ is indeed  half of the Beta function $\mathcal{B}(a,a)$. See the details in Remark \ref{Rem:Beta} in the Appendix. 
		\end{Remark}
		From Theorem \ref{Thm:1.01} and Theorem \ref{Thm:1.02}, we get the following Bray-Miao-type mass-to-capacity inequality and capacity-to-area inequality.
		\begin{Theorem}\label{Thm:1.04}
			Let $(M,g)$ be as in  Theorem \ref{Thm:1.02}.
			Let $k\in (-1, 1]$ be such that
			$$
			1 - \frac{1}{16\pi} \int_{\S} H^{2}  = \frac{4k}{(1+k)^{2}}.
			$$
			Then we have
			\begin{eqnarray}
				&&\mathfrak{m}_{ADM} \ge 2 \operatorname{sgn}(k) \left(I_a(k)\mathfrak{c}_{p}\right)^{\frac{1}{a}},\label{geom-ineq-4}\\
				&&\sqrt{\frac{|\Sigma|}{16\pi}} \ge \frac{(1+k)^{2}}{2|k|} \left( I_{a}(k) \mathfrak{c}_{p} \right)^{\frac{1}{a}}.\label{geom-ineq-5}
			\end{eqnarray}
			Moreover,  equality in each of the above inequalities holds if and only if $(M, g)$ is isometric to the spatial Schwarzschild manifold of mass $m=2 \operatorname{sgn}(k) \left(I_a(k)\mathfrak{c}_{p}\right)^{\frac{1}{a}}$ outside a rotationally symmetric ball, $(\mathcal{M}_{m, r_0}^{3},  g_{m})$ with $r_0=|k|^{-1}\left(I_a(k)\mathfrak{c}_{p}\right)^{\frac{1}{a}}$ when $k\in (0,1]$ and $(M, g)$ is isometric to $(\mathbb{R}^3\setminus B_{r_0}(0), \delta)$ when $k=0$.
		\end{Theorem}

		When $k=1$, we have the following
		\begin{Theorem}\label{Thm:1.03}
			Let $(M,g)$ be as in  Theorem \ref{Thm:1.02}. Assume in addition that $\p M=\S$ is minimal.
			Then we have
			\begin{eqnarray}
				&&\mathfrak{m}_{ADM} \ge 2 \left( I_{a}(1) \mathfrak{c}_{p} \right)^{\frac{1}{a}}, \label{geom-ineq-6}\\
				&&\sqrt{\frac{|\Sigma|}{16\pi}} \ge 2 \left( I_{a}(1) \mathfrak{c}_{p} \right)^{\frac{1}{a}}. \label{geom-ineq-7}
			\end{eqnarray}
			Moreover, equality in each of the above inequalities holds if and only if $(M, g)$ is isometric to the  spatial Schwarzschild manifold of mass $m=2 \left( I_{a}(1) \mathfrak{c}_{p} \right)^{\frac{1}{a}}$, $(\mathcal{M}_{m}^{3},  g_{m})$.
		\end{Theorem}
		\begin{Remark}\
			
			\begin{itemize}
				\item[(i)] We regard $r_{0}=(\frac{\mathfrak{c}_{p}}{a})^{\frac{1}{a}}$ when $k=0$. The inequalities \eqref{geom-ineq-4} and \eqref{geom-ineq-5} still hold true in Theorem \ref{Thm:1.04} for $k=0$ with  equality holding if and only if $(M, g)$ is isometric to the standard Euclidean space outside a rotationally symmetric ball, $(\mathbb{R}^{3}\setminus B_{r_{0}},\delta)$. See Remark \ref{rmk1.3} (iii) below for a detailed explanation. 
				\item[(ii)] For $p=2$, Theorem \ref{Thm:1.03} and Theorem \ref{Thm:1.04} have been first proved by Bray \cite{Br} and Bray-Miao \cite{BM} via Huisken-Ilmanen's weak inverse mean curvature flow. In particular, \eqref{geom-ineq-4} reduces to \eqref{mass-to-capacity}. Miao \cite{M1} gave a new proof by using harmonic functions. Moreover,  Miao \cite{M1} removed the assumption on the nonnegativity of Hawking mass.
				\item[(iii)] For general $p\in (1, 3)$, under the additional assumption of nonnegative Hawking mass, Theorem \ref{Thm:1.03} and Theorem \ref{Thm:1.04}  have been first proved by Xiao \cite{Xiao} via Huisken-Ilmanen's weak  inverse mean curvature flow, following the method of Bray \cite{Br} and Bray-Miao \cite{BM}. Recently,  Benatti-Fogagnolo-Mazzieri also got the Bray-Miao and Xiao's mass-to-capacity inequality in terms of Huisken's isoperimetric mass, see \cite[Theorem 2.15]{BFM1}. 
				\item[(iv)] As $p\rightarrow1$, inequality \eqref{geom-ineq-6} reduces to the Riemannian Penrose inequality \eqref{RPI}.
				
			\end{itemize}
		\end{Remark}
		
		We prove Theorem \ref{Thm:1.01} via exhibiting monotonic quantity for $p$-harmonic functions.
		To illustrate the monotonicity quantity, we shall use the following three one-variable functions. For any $k\in(-1,0)\cup(0,1]$,
		let $m=2 \operatorname{sgn}(k) \left(I_a(k)\mathfrak{c}_{p}\right)^{\frac{1}{a}}$. We set
		\begin{equation}\label{ode-solution1}
			\begin{aligned}
				\alpha(t)
				=& t\left(1+\frac{m}{2t}\right)^{2}
				\left\{ \left( C_{2}\mathfrak{c}_{p}+C_{1}\frac{a}{m} \frac{I_{a}(\frac{m}{2t})}{I_{a}(k)}  \right)
				\eta(t)-C_{1}\frac{1}{m} \right\},\\
				\beta(t)
				=&- \eta(t)\alpha(t)
				+ \left( C_{2}\mathfrak{c}_{p}\frac{m}{a}+C_{1} \frac{I_{a}(\frac{m}{2t})}{I_{a}(k)} \right) \mathfrak{c}_{p}^{-2} t^{2a}\left( 1+\frac{m}{2t}\right)^{4a},\\
				\gamma(t)
				=&- \mathfrak{c}_{p}^{2} t^{-2a}\left( 1+\frac{m}{2t}\right)^{-4a} \eta(t) \alpha(t)
				-  \left( C_{2}\mathfrak{c}_{p}\frac{m}{a}+C_{1} \frac{I_{a}(\frac{m}{2t})}{I_{a}(k)} \right),
			\end{aligned}
		\end{equation}
		where $C_1, C_2\in \mathbb{R}$ and $\eta$ is given by\begin{eqnarray}
			\eta(t)&:=& \mathfrak{c}_{p}^{-1}t^{a}\left(1 + \frac{m}{2 t}\right)^{2a-1}\left(1- \frac{m}{2t}\right). \label{equ:eta}
		\end{eqnarray}
		
		We remark that the three one-variable functions $\alpha, \beta, \gamma$ satisfy the following system of ordinary differential equations:  (See Proposition \ref{ode-solution})
		\begin{equation}\label{equ: differential equations}
			\begin{cases}
				\alpha'(t) - (2 a + 1) \eta(t) f'(t) \alpha(t) - a f'(t) \beta(t) = 0, \\
				\beta'(t) + (2 a + 1) (\eta(t))^2 f'(t) \alpha(t) = 0, \\
				\gamma'(t) = - f'(t) \alpha(t),
			\end{cases}
		\end{equation}
		where
		\begin{eqnarray}
			f (t)&:=& 1 - \int_{t}^{\infty} \mathfrak{c}_{p} s^{-a-1}\left(1+\frac{m}{2s}\right)^{-2a}ds
			=1-\frac{I_{a}(\frac{m}{2t})}{I_{a}(k)}.\label{equ:f}
		\end{eqnarray}
		
		We have the following monotone quantity along regular level sets of $p$-harmonic functions.
		\begin{Theorem}\label{thm:Monot}
			Let $(M,g)$ and $u$ be as in Theorem \ref{Thm:1.02}.
			For any $k\in(-1,0)\cup(0,1]$, let $\a, \b, \g$ be three one-variable functions given by \eqref{ode-solution1} with
			\begin{eqnarray}\label{alpha>0-assumpt}
				C_{2} \ge 0, \hbox{ and }
				\left( C_{2}\mathfrak{c}_{p}+C_{1}\frac{a}{m} \right)
				\eta(r_{0}) \ge C_{1}\frac{1}{m}. 	
			\end{eqnarray}
			Let $F: [r_{0}, \infty)\to \mathbb{R}$ be given by
			\begin{equation} \label{equ: F}
				F(t):= 4 \pi \gamma(t) + \alpha(t) \int_{\Sigma_t} H |\nabla u|   + \beta(t) \int_{\Sigma_t} |\nabla u|^2  .
			\end{equation}
			where $\S_t$ is a regular level set of $u$ given by $$\Sigma_t=\{x\in M| u(x)=f(t)\}.$$
			Then $F(t)$ is monotone non-increasing on $$\mathcal{T}:=\left\{ t\in\left[r_{0},\infty \right)\Big| f(t) \hbox{ is a regular value of }u \right\},$$ that is, for $t_{1}, t_{2}\in \mathcal{T}$, $t_1<t_2$, we have $F(t_{1}) \ge F(t_{2})$. Moreover, $F$ is a constant on $\mathcal{T}$ if and only if $(M, g)$ is isometric to the spatial Schwarzschild manifold of mass $m$ outside a rotationally symmetric ball, $(\mathcal{M}_{m, r_0}^{3},  g_{m})$.
		\end{Theorem}
		
		\begin{Remark}\label{rmk1.3}\
			
			\begin{itemize}
				\item[(i)] The condition \eqref{alpha>0-assumpt} is used to make sure that $\alpha(t)\ge0$ on $[r_{0},+\infty)$, which is the key to  the monotonicity of $F(t)$. See the details in Proposition \ref{alpha>0} in the Appendix.  
				\item[(ii)] For different choice of $C_1$ and $C_2$, we may get \eqref{geom-ineq-10} and \eqref{geom-ineq-20} respectively, see Section 6 for detailed proof.
				\item[(iii)] Our proof also works for $k=0$. In fact, it becomes much easier in this case. Precisely, we have $m=0$, $\eta(t)=\mathfrak{c}_{p}^{-1}t^{a}$ and $f'(t)=\mathfrak{c}_{p}t^{-a-1}$. One may solve \eqref{equ: differential equations} to get
				\begin{equation}\label{ode-solution2}
					\begin{aligned}
						\alpha(t)
						=&C_2t^{a+1}-C_1\frac{1}{a+1},\\
						\beta(t)
						=&C_1\frac{2a+1}{a(a+1)}\mathfrak{c}_{p}^{-1}t^a-C_2\mathfrak{c}_{p}^{-1}t^{2a+1},\\
						\gamma(t)
						=&- C_2\mathfrak{c}_{p}t-C_1\mathfrak{c}_{p}\frac{1}{a(a+1)}t^{-a}.
					\end{aligned}
				\end{equation}
				When we choose $C_2=\mathfrak{c}_{p}^{-1}$ and $C_1=0$, 	we get
				$$F_1(t)=-4\pi t+t^{a+1} \mathfrak{c}_{p}^{-1}\int_{\Sigma_t} H |\nabla u|-\mathfrak{c}_{p}^{-2}t^{2a+1}\int_{\Sigma_t} |\nabla u|^2.$$
				This is the monotone quantity found by Agostiniani-Mantegazza-Mazzieri-Oronzio \cite[(1.10)]{AMMO}.	
				When we choose $C_2=0$ and $C_1=-a(a+1)\mathfrak{c}_{p}^{-1}$, we get
				$$F_2(t)=4\pi t^{-a}+a  \mathfrak{c}_{p}^{-1}\int_{\Sigma_t} H |\nabla u|-(2a+1)\mathfrak{c}_{p}^{-2}t^a \int_{\Sigma_t} |\nabla u|^2.$$
				This is the monotone quantity found by Hirsch-Miao-Tam \cite[(1.19)]{HMT}
				Both monotone quantities are constants if and only if $(M, g)$ is isometric to $\mathbb{R}^3$ outside a ball.
			\end{itemize}
		\end{Remark}

		Next we make some comments on our approach to find the monotone quantities. As we have mentioned before, in a series of previous {significant} works \cite{AMO, AMMO, M1, O, HMT}, various monotone quantities associated with $p$-capacitary functions have been found in $3$-mainfolds with nonnegative scalar curvature. These works share a common feature that when a first explicit monotone quantity has been found, some other monotone quantities might be found by making use of the first explicit quantity as well as its asymptotic behavior. Our paper gives a new perspective. We consider general quantities with undetermined coefficients {associated with  Schwarzschild manifold model}. The monotonicity of the general quantities can be verified if the undetermined coefficients satisfy a system of ODEs. {We solve out the system of ODEs so that we get all the possible monotone quantities and subsequently employ them to deduce new results which generalized Miao \cite{M1} and Oronzio's \cite{O} results  in the case $p=2$.}


		The rest of our paper is organized as follows. In Section 2, we study the $p$-capacitary function in a spatial Schwarzschild manifold outside a rotationally symmetric sphere. In Sections 3 and 4, we prove the monotonicity and rigidity part in Theorem \ref{thm:Monot} respectively. In Section 5, we study the asymptotic behavior for the monotone quantities and in Section 6, we given some applications. In the Appendix, we solve out the solutions to a system of ODEs that appears in the proof of monotonicity. 
		
		\
		
		\noindent{\bf Acknowledgements.} We are grateful to Prof. Pengzi Miao for his interest  in this work and his useful comments. When our paper is being finalized, we found a recent arXiv preprint ``The Sharp $p$-Penrose Inequality" by Liam Mazurowski and Xuan Yao (arXiv:2305.19784), which proves  the mass-to-$p$-capacity inequality for manifolds with minimal boundary via related techniques.
		
		\
		
		\section{$p$-capacity in Schwarzschild manifolds}
		
		Given $m\in\mathbb{R}$ and $r_0 > \frac{|m|}{2}$. Let  $(\mathcal{M}_{m, r_0}^{3},   g_{m})$ be the spatial Schwarzschild manifold  of mass $m$ outside a rotationally symmetric sphere of mass $m$, given by \eqref{schwarzschild}. Here $m$ can be taken to be a negative real number, we still call it the Schwarzschild manifold (See Bray-Jauregui \cite{BJ} and Miao \cite{M2}).  Now we calculate the $p$-capacitary function in $(\mathcal{M}_{m, r_0}^{3},   g_{m})$.
		\begin{Proposition} \label{Appendix: basic_solution}
			Let   $(\mathcal{M}_{m, r_0}^{3},   g_{m})$ be the spatial Schwarzschild manifold of mass $m$  outside a rotationally symmetric sphere given by \eqref{schwarzschild} and $u$ be the solution to \eqref{p-laplace}. Then  $u(x)= f_{m, r_0}(r), r=|x|$, where $f_{m, r_0}$ is given by
			\begin{eqnarray*}
				f_{m, r_0}(r)&:=& 1 - \int_{r}^{\infty} \mathfrak{c}_{p} s^{-a-1}(1+\frac{m}{2s})^{-2a}ds
				=1-\frac{I_{a}(\frac{m}{2r})}{I_{a}(\frac{m}{2r_{0}})}, 	\end{eqnarray*}
			where $\mathfrak{c}_{p}$ and $m$ can be related by 
			\begin{equation} \label{equ:m}
				|m| = 2 \left( \mathfrak{c}_{p}I_{a}(\frac{m}{2r_{0}}) \right)^{\frac{1}{a}} .
			\end{equation}
		\end{Proposition}
		\begin{proof}
			Since $(\mathcal{M}_{m, r_0}^{3},   g_{m})$ is rotationally symmetric and the solution to  \eqref{p-laplace} is unique, $u$ is rotationally symmetric, i.e. there exists some $f(r)$ such that $u(x)=f(r)$.
			The Euclidean $p$-Lapacian on $u$ gives
			$$
			\bar{\Delta}_{p} u
			= \left( f'(r) \right)^{p - 2} \left[ (p - 1) f''(r) + 2 \frac{1}{r}f'(r) \right].
			$$
			Let $w = 1 + \frac{m}{2 r}$ so that $ g_{m}=w^4\bar g$. By using the transformation formula for $p$-Laplacian under conformal change,
			we get that
			\begin{eqnarray*}
				\Delta_{p} u&=& w^{- 2p} \left( \bar{\Delta}_{p} u + 2(3 - p) |\bar{\nabla} u|^{p - 2}_{\bar{g}} \<\bar{\nabla}u, \bar\nabla\ln w\> \right)
				\\&=& w^{- 2p} (f'(r))^{p - 2} \left[ (p - 1) f''(r) + 2 \frac{1}{r}f'(r) + 2(3 - p) f'(r) \partial_{r}(\operatorname{ln} w) \right].
			\end{eqnarray*}
			It follows from $\Delta_{p} u = 0$ that
			$$f'(r) = C r^{-a-1}\left(1+\frac{m}{2r}\right)^{-2a},$$ where $C$ is a positive constant. We claim that $C=\mathfrak{c}_{p}$. Indeed,
			since $$|\nabla u|_{g} =  w^{-2} |\bar{\nabla} u|_{\bar{g}} = w^{-2} f'(r),$$ we see
			$$
			\operatorname{Area}_{g} (\partial \mathcal{M}_{m, r_0})
			= \int_{\{ r = r_{0} \}}  d \sigma
			= \int_{\{ r = r_{0} \}} w^{4} d \bar{\sigma}
			= (w(r_{0}))^{4} 4 \pi r_{0}^{2}.
			$$
			It follows that
			\begin{eqnarray*}
				\mathfrak{c}_{p}= \left( \frac{1}{4 \pi} \int_{\partial \mathcal{M}_{m, r_0}} |\nabla u|_{g}^{p-1} d \sigma \right)^{\frac{1}{p-1}}
				= (w(r_{0}))^{ 2\frac{3-p}{p-1}} r_{0}^{\frac{2}{p-1}} f'(r_{0})
				= C.
			\end{eqnarray*}

			Since $u \rightarrow 1$ as $r \rightarrow \infty$, we get
			\begin{equation*}
				f(r) = 1 - \int_{r}^{+ \infty} \mathfrak{c}_{p} t^{-a-1}\left(1+ \frac{m}{2t}\right)^{-2a} d t.
			\end{equation*}
			By a change of variable $s=\frac{|m|}{2t}$, we see that 
			\begin{equation*}
				f(r) = 1 - (\frac{|m|}{2})^{-a} \mathfrak{c}_{p}\int_{0}^{\frac{|m|}{2r}}  s^{a-1}(1+\operatorname{sgn}(m)s)^{-2a} d s=1 - (\frac{|m|}{2})^{-a} \mathfrak{c}_{p}I_{a}(\frac{m}{2r}).
			\end{equation*}
			Since $u = 0$ on boundary $\partial \mathcal{M}_{m, r_0} = \{r=r_{0}\}$, we get
			$$
			(\frac{|m|}{2})^{-a} \mathfrak{c}_{p}I_{a}(\frac{m}{2r_0})=1.
			$$
			
		\end{proof}
		
		\begin{Proposition}
			Let   $(\mathcal{M}_{m, r_0}^{3},   g_{m})$ be a spatial Schwarzschild manifold outside a rotationally symmetric sphere of mass $m$ given by \eqref{schwarzschild} and $u$ be the solution to \eqref{p-laplace}. Let $H$ be the mean curvature of the level set $S_r=\{|x|=r\}$. Then
			\begin{eqnarray*}
				H=\frac{2}{r}\left(1 + \frac{m}{2 r}\right)^{-3}\left(1- \frac{m}{2r}\right), 	\end{eqnarray*}
			\begin{equation}
				\frac{H}{2|\nabla u|_{g}}=\eta(r)= \mathfrak{c}_{p}^{-1}r^{a}\left(1 + \frac{m}{2 r}\right)^{2a-1}\left((1- \frac{m}{2r}\right).
			\end{equation}
		\end{Proposition}
		\begin{proof}
			Denote $\bar H$ and $\bar \nu$ be the mean curvature and the unit normal of $S_r=\{|x|=r\}$ under the Euclidean metric $\bar g$, respectively.
			By the transformation formula for mean curvature under conformal change, we have	$$
			H = w^{ - 3 }\left( \bar{H} w + 4 \frac{\partial w}{\partial \bar{\nu}} \right) = w^{ - 3 } \left(  \frac{2}{r} w + 4 \frac{\partial w}{\partial r} \right)  = \frac{2}{r}\left(1 + \frac{m}{2 r}\right)^{-3}\left(1- \frac{m}{2r}\right).
			$$
			Note that $|\nabla u|_{g} = w^{-2} f'(r),$ the second assertion follows.
		\end{proof}
		

		\section{Monotonicity}
		
		The aim of this section is to prove the monotonicity of $F(t)$ in Theorem \ref{thm:Monot}.	
		\subsection{Monotonicity of $F(t)$ when $|\nabla u|\neq 0$}
		We assume that $|\nabla u|\neq 0$ in $M$, then $u\in C^\infty(M)$.
		Consider the level set $\S_t=\{u=f(t)\}$, where $f$ is a given one-variable function. One sees readily that $\{\S_t\}$ satisfies the flow equation
		\begin{equation} \label{meancurv-evol}
			\begin{cases}
				\Psi: \Sigma \times (r_0, + \infty) \rightarrow \ M, \\
				\partial_t \Psi(p, t)=f'(t)\frac{\nabla u}{|\nabla u|^2}=f'(t)\frac{1}{|\nabla u|}\nu.
			\end{cases}
		\end{equation}
		The following basic facts are well-known, see for example \cite{HMT}.
		\begin{Lemma}
			The mean curvature of a regular level set $\S_t$ is given by
			\begin{eqnarray}
				&&H= (1-p)\frac{1}{ |\nabla u|}u_{\nu\nu}=\frac{1-p}{2}\frac{1}{|\nabla u|^{2}}\nu(|\nabla u|^2), \label{equ: H}
			\end{eqnarray}	
			where $u_{\nu\nu}=\n^2 u \left( \nu, \nu \right)$.
		\end{Lemma}
		The evolution equation for the mean curvature along the flow \eqref{meancurv-evol} is as following, see for example \cite{M1}.
		\begin{Lemma}\label{lem:2.3}
			\begin{equation} \label{equ: dH_t}
				\frac{\p}{\p t} H
				= - f'(t) \left( \Delta_{\Sigma_t} \left( \frac{1}{|\nabla u|} \right)
				+ \left( |h|^2 + \operatorname{Ric}_M (\nu, \nu) \right)\frac{1}{|\nabla u|} \right).
			\end{equation}	
		\end{Lemma}
		
		Next we prove the following variational formula.
		\begin{Lemma}\label{lem:2.4} Along the level set flow, we have
			\begin{eqnarray}
				&&\frac{d}{d t}  \int_{\Sigma_t} |\nabla u|^2
				= - a f'(t)  \int_{\Sigma_t} H |\nabla u| ,\label{equ: d B}\\
				&&\frac{d}{d t} \int_{\Sigma_t} H |\nabla u|
				= -f'(t) \left\{  \int_{\Sigma_t} |\nabla u|^{-2}|\nabla_{\Sigma_t} |\nabla u||^2 + \frac{1}{2}( R_{M} - K_{\Sigma_t}+ |\overset{\circ}{h}|^2) + \frac{2 a + 1}{4}H^2 \right\},\label{equ: d A}
			\end{eqnarray}	
			where $\overset{\circ}{h}$ denotes the traceless part of the second fundamental form of the level set.
		\end{Lemma}
		\begin{proof}
			Recall that the variation field of the level set flow is $\partial_t \Psi = f'(t) \frac{1}{|\nabla u|} \nu $.
			Using \eqref{equ: H}, one computes
			\begin{align*}
				\frac{d}{d t}  \int_{\Sigma_t} |\nabla u|^2
				=& \int_{\Sigma_t}  f'(t) \frac{1}{|\nabla u|} \nu (|\nabla u|^2)+ |\nabla u|^2 H f'(t)  \frac{1}{|\nabla u|} \\
				=& \int_{\Sigma_t}  f'(t) \frac{2}{1 - p} |\nabla u| H + f'(t) H |\nabla u|  \\
				=& - af'(t)  \int_{\Sigma_t} H |\nabla u|.
			\end{align*}
			On one hand, by the divergence theorem, we obtain
			\begin{align}\label{equ: div}
				\int_{\Sigma_t} |\nabla u| \Delta_{\Sigma_t} \left( \frac{1}{|\nabla u|} \right)
				=& \int_{\Sigma_t} |\nabla u|^{-2} \left| \nabla_{\Sigma_t} {|\nabla u|} \right|^{2}  .
			\end{align}
			On the other hand, the Gauss equation tells that
			\begin{equation}\label{equ:gauss}
				2 \operatorname{Ric}_M (\nu, \nu) = R_{M} - K_{\Sigma_t} + H^{2} - |h|^{2}.
			\end{equation}
			It follows from \eqref{equ: H}, \eqref{equ: dH_t}, \eqref{equ: div} and \eqref{equ:gauss} that
			\begin{align*}
				\frac{d}{d t} \left(  \int_{\Sigma_t} H |\nabla u|    \right)
				=& \int_{\Sigma_t} - f'(t) \left( \Delta_{\Sigma_t} \left( \frac{1}{|\nabla u|} \right) + \left( |h|^2 + \operatorname{Ric}_M (\nu, \nu)  \right)\frac{1}{|\nabla u|} \right) |\nabla u| \\
				&+ \int_{\Sigma_t} H f'(t) \frac{1}{|\nabla u|} \nu (|\nabla u|) + H|\nabla u| Hf'(t)   |\nabla u|^{- 1} \\
				=& -f'(t) \left\{  \int_{\Sigma_t} |\nabla u|^{-2}|\nabla_{\Sigma_t} |\nabla u||^2 + \frac{1}{2}( R_{M} - K_{\Sigma_t}+ |\overset{\circ}{h}|^2) + \frac{2 a + 1}{4}H^2 \right\}.
			\end{align*}
		\end{proof}
		
		\begin{Proposition}
			Let $(M, g)$, $u$, $F(t)$ be as in Theorem \ref{thm:Monot}.  Assume in addition that $|\nabla u|\neq 0$ in $M$. Then $F'(t)\le 0$.
		\end{Proposition}
		\begin{proof}
			Since $|\nabla u| \neq 0$, it  follows from Lemma \ref{lem:2.4} that
			\begin{align*}
				F'(t)=& 4 \pi \gamma'(t) + \alpha'(t) \int_{\Sigma_t} H |\nabla u|   + \beta'(t) \int_{\Sigma_t} |\nabla u|^2   \\
				&- \alpha(t) f'(t) \left\{  \int_{\Sigma_t} |\nabla u|^{-2}|\nabla_{\Sigma_t} |\nabla u||^2 + \frac{1}{2}( R_{M} - K_{\Sigma_t}+ |\overset{\circ}{h}|^2) + \frac{2 a + 1}{4}H^2 \right\} \\
				&-a \beta(t) f'(t) \int_{\Sigma_t} H |\nabla u|   \\
				=& 4 \pi \gamma'(t) - \alpha(t) f'(t) \left\{  \int_{\Sigma_t} |\nabla u|^{-2}|\nabla_{\Sigma_t} |\nabla u||^2 + \frac{1}{2} ( R_{M} - K_{\Sigma_t} + |\overset{\circ}{h}|^2 ) \right\} \\
				&- (2 a + 1) \alpha(t) f'(t) \int_{\Sigma_t} \left( \frac{H}{2} - \eta(u) |\nabla u| \right)^2 \\
				&+ \left[\alpha'(t) - (2 a + 1) \eta(t) f'(t) \alpha(t) - a f'(t) \beta(t)\right]\int_{\Sigma_t} H |\nabla u|    \\
				&+ \left[\beta'(t)  + (2 a + 1) (\eta(t))^2 f'(t) \alpha(t)\right]\int_{\Sigma_t} |\nabla u|^2  .
			\end{align*}
			Using the 	system of ODEs \eqref{equ: differential equations}, we get
			\begin{align*}
				F'(t)			=&\gamma'(t) \left\{  \int_{\Sigma_t} |\nabla u|^{-2}|\nabla_{\Sigma_t} |\nabla u||^2 + \frac{1}{2}R_{M}  +\frac12 |\overset{\circ}{h}|^2 + (2 a + 1) \left( \frac{H}{2} - \eta(u) |\nabla u| \right)^2 \right\}\\
				&+\gamma'(t)\left(4\pi-\int_{\Sigma_t}  K_{\Sigma_t}\right).
			\end{align*}
			By the assumption $H_2(M, \Sigma) = 0$ and $\S$ is connected, we known that $\S_t$ is connected, see for example \cite{M1}. It follows from the Gauss-Bonnet formula that $\int_{\Sigma_t} K_{\Sigma_t} \leq 4\pi$. In view of Proposition \ref{alpha>0} and the assumption \eqref{alpha>0-assumpt}, we see that $\alpha(t)\ge 0$. The assertion follows since $\g'(t)=-f'(t)\a(t)\le 0$, $R_M\ge 0$ and $2a+1=\frac{5-p}{p-1}>0$. 	\end{proof}
		
		\subsection{Monotonicity of $F(t)$ via regularization}
		In order to prove the monotonicity part in Theorem \ref{thm:Monot}, we need to establish the monotone property of $F(t)$ via regularization. Let $u$ be the solution to \eqref{p-laplace}.
		In the following we denote $$f_0(t)=f(t), \eta_0(t)=\eta(t), \a_0(t)=\a(t), \b_0(t)=\b(t), \g_0(t)=\g(t).$$
		Following \cite{AMMO,HMT}, we approximate $u$ by smooth function $\{v_{\ve}\}_{\ve>0}$, which is a sequence of solutions of
		\begin{equation}\label{eqn:3.1}
			\left\{
			\begin{aligned}
				{\rm div}(|\n v_\ve|_{\ve}^{p-2}\n v_\ve)&=0 \ \ \  \  {\rm in}\ \ M(T),\\
				v_\ve&=0\ \ \ \ {\rm on} \ \ \p M,\\
				v_\ve&=f_0(T)\ \ \  \ {\rm on}\  \ \Sigma(T),
			\end{aligned}\right.
		\end{equation}
		where $M(T)=\{0<u<f_0(T)\}$, $\Sigma(t)=\{u=f_0(t)\}$ and $|\n v_\ve|_{\ve}=\sqrt{|\n v_\ve|^2+\ve^2}$. It is clear that for any $\ve>0$, $v_{\ve}$ is smooth. By \cite{DE1,DE2},  as $\ve\rightarrow 0$, $v_{\ve}\rightarrow u$ in $C^{1,\beta}$-topology  for some $\beta >0$ on any compact subsets of $M(T)$, and $v_{\ve}\rightarrow u$ in $C^{\infty}$-topology on any compact subsets of $M(T)\setminus\{|\n u|\neq 0\}$.
		
		In the following, for notation simplicity, we drop the subscription $\ve$ and simply write $v$ for $v_\ve$,  when no confusion occurs. We define
		\begin{equation}\label{eqn:3.2}
			\left\{
			\begin{aligned}
				&{\rm Cap}_{p,\ve}=\int_{\p M}|\n v|_{\ve}^{p-2}|\n v|=\int_{\Sigma_{t, \ve}}|\n v|_{\ve}^{p-2}|\n v|,\\
				&\mathfrak{c}_{p,\ve}=\left(\frac{{\rm Cap}_{p,\ve}}{4\pi}\right)^{\tfrac1{p-1}},\\
				&m_{\ve}=2 \operatorname{sgn}(k) \left( \mathfrak{c}_{p,\ve}I_{a}(k) \right) ^{\frac{1}{a}}, \\
				&f_{\ve}(t)=1 - \int_{t}^{\infty} \mathfrak{c}_{p,\ve} s^{-a-1}\left(1+\frac{m_{\ve}}{2s}\right)^{-2a},\\
				&\eta_{\ve}(t) = \mathfrak{c}_{p,\ve}^{-1}t^{a}\left(1 + \frac{m_{\ve}}{2 t}\right)^{2a-1}\left(1- \frac{m_{\ve}}{2t}\right).
			\end{aligned}\right.
		\end{equation}
		Denote $\Sigma_{t, \ve}:=\{v_{\ve}=f_{\ve}(t)\}$. When $\Sigma_{t, \ve}$ is a regular hypersurface, we define $F_{\ve}(t)$ as follows:
		\begin{equation}\label{eqn:3.3}
			F_{\ve}(t) = 4 \pi \gamma_{\ve}(t) + \alpha_{\ve}(t) \int_{\Sigma_{t, \ve}} H |\nabla v| + \beta_{\ve}(t) \int_{\Sigma_{t, \ve}} |\nabla v|^2.
		\end{equation}
		Here $\alpha_{\ve}(t)$, $\beta_{\ve}(t)$, $\gamma_{\ve}(t)$ 
		are solutions to the corresponding systems of ODEs:
		\begin{equation}\label{eqn:3.4}
			\left\{
			\begin{aligned}
				&0=\alpha_{\ve}'(t) - (2 a + 1) \eta_{\ve}(t) f'_{\ve}(t) \alpha_{\ve}(t) - a f'_{\ve}(t) \beta_{\ve}(t), \\
				&0=\beta_{\ve}'(t) + (2 a + 1) (\eta_{\ve}(t))^2 f'_{\ve}(t) \alpha_{\ve}(t), \\
				&0=\gamma'_{\ve}(t) + f'_{\ve}(t) \alpha_{\ve}(t).
			\end{aligned}\right.
		\end{equation}
		By the similar consideration in Propositions \ref{ode-solution} and \ref{alpha>0}, one sees that $\alpha_\ve(t)$ is given by
		\begin{equation*}
			\begin{aligned}
				\alpha_\ve(t)
				=& t\left(1+\frac{m_\ve}{2t}\right)^{2}
				\left\{ \left( C_{2,\ve}\mathfrak{c}_{p,\ve}+C_{1,\ve}\frac{a}{m_\ve} \frac{I_{a}(\frac{m_\ve}{2t})}{I_{a}(\frac{m_\ve}{2r_0})}  \right)
				\eta_\ve(t)-C_{1,\ve}\frac{1}{m_\ve} \right\},
			\end{aligned}
		\end{equation*}
		where $C_{1,\ve}$ and $C_{2,\ve}$ are two constants,
		and $\alpha_\ve(t)\ge 0$ if and only if \begin{eqnarray}\label{alpha>0-assumpt1}
			C_{2,\ve} \ge 0, \hbox{ and }
			\left( C_{2,\ve}\mathfrak{c}_{p,\ve}+C_{1,\ve}\frac{a}{m_\ve} \right)
			\eta_\ve(r_{0}) \ge C_{1,\ve}\frac{1}{m_\ve}. 	
		\end{eqnarray}	
		By the assumption \eqref{alpha>0-assumpt} and the simple fact that $\mathfrak{c}_{p,\ve}, m_\ve, \eta_\ve(r_{0})$ converge to $\mathfrak{c}_{p}, m, \eta(r_{0})$, respectively, as $\ve\to 0$, we may choose appropriate $C_{i,\ve}, i=1,2,$ so that \eqref{alpha>0-assumpt1} holds and $C_{i,\ve}\to C_i, i=1,2,$ as  $\ve\to 0$.
		For such choice, we see $\alpha_\ve(t)\ge 0$. This fact will be used in the following.
		
		Abuse of notation, we use $\alpha_{\ve}(v),\beta_{\ve}(v), \gamma_{\ve}(v), \eta_{\ve}(v)$ to indicate $\alpha_{\ve}(t), \beta_{\ve}(t), \gamma_{\ve}(t), \eta_{\ve}(t)$ for $t=f_{\ve}^{-1}(v)$, respectively. Thus $\alpha_{\ve}(v),\beta_{\ve}(v), \gamma_{\ve}(v)$,  as functions of $v$,  satisfy	that 
		\begin{equation}\label{eqn:3.5}
			\left\{
			\begin{aligned}
				&0=\alpha_{\ve}'(v) - (2 a + 1) \eta_{\ve}(v) \alpha_{\ve}(v) - a \beta_{\ve}(v), \\
				&0=\beta_{\ve}'(v) + (2 a + 1) (\eta_{\ve}(v))^2  \alpha_{\ve}(v), \\
				&0=\gamma'_{\ve}(v) +  \alpha_{\ve}(v).
			\end{aligned}\right.
		\end{equation}
		It is easy to see that
		\begin{eqnarray}\label{equ-Delta u}
			\Delta v=(2-p)\frac{|\n v|^2}{|\n v|^2_{\ve}}v_{\nu\nu}, 
		\end{eqnarray}
		where $v_{\nu\nu}=\frac{g(\nabla|\nabla v|,\nabla v)}{|\nabla v|}$, and the mean curvature $H$ of $\Sigma_{t, \ve}$ is given by
		\begin{eqnarray}\label{equ-H}
			H=\frac1{|\n v|}(\Delta v-v_{\nu\nu})=-\frac1{|\n v|}\frac{(p-1)|\n v|^2+\ve^2}{|\n v|^2_{\ve}}v_{\nu\nu}. 
		\end{eqnarray}
		Using \eqref{eqn:3.2} and \eqref{equ-H}, we can write
		\begin{equation*}
			\begin{aligned}
				F_{\ve}(t) 
				=&\int_{\Sigma_{t, \ve}}4\pi\gamma_{\ve}(v) {\rm Cap}_{p,\ve}^{-1}|\n v|^{p-2}_{\ve}|\n v|+\alpha_{\ve}(v)(\Delta v-v_{\nu\nu})+\beta_{\ve}(v)|\n v|^2.
			\end{aligned}
		\end{equation*}

		Let
		$
		X_{\ve}=U_{\ve}+V_{\ve}+W_{\ve}
		$
		where
		\begin{equation*}
			\left\{
			\begin{aligned}
				U_{\ve}=&4\pi\gamma_{\ve}(v) {\rm Cap}_{p,\ve}^{-1}|\n v|^{p-2}_{\ve}\n v,\\
				V_{\ve}=&\alpha_{\ve}(v)\left(\frac{\Delta v}{|\n v|}\n v-\n |\n v|\right),\\
				W_{\ve}=&\beta_{\ve}(v)|\n v|\n v.
			\end{aligned}\right.
		\end{equation*}	
		Then
		\begin{equation}\label{eqn:3.6}
			F_{\ve}(t)=\int_{\Sigma_{t, \ve}}\left\langle X_{\ve},\, \frac{\n v}{|\n v|}\right\rangle.
		\end{equation}
		By adapting the proof of \cite[Lemma 1.3]{AMMO}), since $\alpha_{\ve}(t), \beta_{\ve}(t), \gamma_{\ve}(t)$ converge to $\alpha(t), \beta(t), \gamma(t)$, respectively, as $\ve\to 0$, we have the following lemma.
		\begin{Lemma}\label{lem:3.2}
			Suppose $\{u=f_{0}(t)\}$ is regular for $f_0(t)\in(0, f_0(T))$. Then for  $\ve>0$ small enough, $\Sigma_{t, \ve}=\{v_{\ve}=f_{\ve}(t)\}$ is also regular. Moreover,
			\begin{equation*}
				\lim_{\ve\rightarrow 0}F_{\ve}(t)=F(t).
			\end{equation*}
		\end{Lemma}
		%
		
		
		For $\delta>0$, let
		\begin{equation*}
			\left\{
			\begin{aligned}
				V_{\ve,\delta}=&\alpha_{\ve}(v)\left(\frac{\Delta v}{|\n v|_{\delta}}\n v-\n |\n v|_{\delta}\right),\\
				W_{\ve,\delta}=&\beta_{\ve}(v)|\n v|_{\delta}\n v,\\
				X_{\ve,\delta}=&U_{\ve}+V_{\ve,\delta}+W_{\ve,\delta}.
			\end{aligned}\right.
		\end{equation*}
		It is clear that $U_{\ve}, V_{\ve,\delta}, W_{\ve,\delta}$ are smooth in $M_T$. Let $t_1 < t_2 $ such that $\Sigma_{t_1, \ve},\Sigma_{t_2, \ve}$ are regular.
		One sees from the divergence theorem and \eqref{eqn:3.6} that
		\begin{equation}\label{eqn:3.7}
			\begin{aligned}
				F_{\ve}(t_2)-F_{\ve}(t_1)=&\lim_{\delta\rightarrow0}\int_{\{f_\ve(t_1)<v<f_\ve(t_2)\}}{\rm div}\ X_{\ve,\delta}.
			\end{aligned}
		\end{equation}
		Next we compute the intergrand in the right hand side of\eqref{eqn:3.7}.
		\begin{Lemma}\label{lem:3.3}\
			\begin{itemize}
				\item[(i)] ${\rm div} U_{\ve}=4\pi \gamma'_{\ve}(v) {\rm Cap}_{p,\ve}^{-1}|\n v|_{\ve}^{p-2}|\n v|^2$.
				
				\item[(ii)] At the points where $|\n v|=0$, we have ${\rm div}W_{\ve,\delta}=\beta_{\ve}(v)\delta\Delta v$.
				
				At the points where $|\n v|>0$, we have
				$$
				{\rm div}W_{\ve,\delta}=\beta_{\ve}(v)\left((2-p)\frac{|\n v|_{\delta}|\n v|^2}{|\n v|^2_{\ve}}v_{\nu\nu}+\frac{|\n v|^2}{|\n v|_{\delta}}v_{\nu\nu}\right)+\beta'_{\ve}(v)|\n v|^2|\n v|_{\delta}.
				$$
				
				\item[(iii)] At the points where $|\n v|=0$,  we have ${\rm div}V_{\ve,\delta}\leq0$.
				
				At the points where $|\n v|>0$,
				we have ${\rm div}V_{\ve,\delta}\leq I_{\ve,\delta}$,
				where
				\begin{equation*}
					\begin{aligned}
						I_{\ve,\delta}=&\alpha_{\ve}(v)|\n v|^{-1}_{\delta}\Bigg\{(2-p)^2\frac{|\n v|^4}{|\n v|_{\ve}^4}v_{\nu\nu}^2-(2-p)\frac{|\n v|^4}{|\n v|_{\ve}^2|\n v|_{\delta}^2}v_{\nu\nu}^2-|\n v|^2{\rm Ric}(\nu,\nu)\\
						&-\frac{1}{2}\frac{|\nabla v|^2}{|\nabla v|_{\ve}^2} \left( (p-2)^2\frac{|\nabla v|^2}{|\nabla v|_{\ve}^2} + 2p-3 \right)v_{\nu\nu}^2\Bigg\}\\
						&+\alpha'_{\ve}(v)\left((2-p)\frac{|\n v|^4}{|\n v|^2_{\ve}|\n v|_{\delta}}-\frac{|\n v|^2}{|\n v|_{\delta}}\right)v_{\nu\nu}.
					\end{aligned}
				\end{equation*}
			\end{itemize}
		\end{Lemma}
		\begin{proof}(i)(ii) By simple computation, using \eqref{equ-Delta u}, we get (i)(ii). Next we prove (iii).
			
			From the proof of Lemma 3.3 (iii) in \cite{HMT}, we can see that
			\begin{equation}\label{xeq11}
				\begin{aligned}
					{\rm div}\left(\frac{\Delta v}{|\n v|_{\delta}}\n v\right)
					=&(2-p)^2\frac{|\n v|^4}{|\n v|^4_{\ve}|\n v|_{\delta}}v_{\nu\nu}^2-(2-p)\frac{|\n v|^4}{|\n v|^2_{\ve}|\n v|_{\delta}^3}v_{\nu\nu}^2+\frac{\langle\n\Delta v,\, \n v\rangle}{|\n v|_{\delta}},
				\end{aligned}
			\end{equation}
			and
			\begin{equation}\label{xeq1-1}
				\begin{aligned}
					{\rm div}(\n|\n v|_{\delta})=&|\n v|^{-1}_{\delta}\left(|\n^2 v|^2+{\rm Ric}(\n v,\n v)\right)-\frac{|\n v|^2}{|\n v|^3_{\delta}}|\n |\n v||^2+|\n v|^{-1}_{\delta}\langle\n v,\, \n\Delta v\rangle.
				\end{aligned}
			\end{equation}
			Recall that  for $\forall X,Y \in T \Sigma_{t, \ve}$, we have
			\[
			\begin{aligned}
				\n^2v (X, Y)= - |\nabla v| g(\nabla_{Y}X, \nu)= - |\nabla v| h (X, Y).
			\end{aligned}
			\]
			It follows that	\[
			|\n^2 v|^2 = |\nabla v|^2 |h|^2 + 2 |\nabla^T|\nabla v||^2 + v_{\nu\nu}^2.
			\]
			Using \eqref{equ-H} and $|\overset{\circ}{h}|^2 = |{h}|^2 - \frac{1}{2}H^2$,  we get the following Kato-type inequality 	
			\begin{equation} \label{eqn:Kato}
				\begin{aligned}
					|\n^2 v|^2
					=& |\nabla v|^2 |\overset{\circ}{h}|^2 + 2|\nabla^{T}|\nabla v||^2
					+ v_{\nu\nu}^2 + \frac{1}{2}|\nabla v|^2 H^2 \\
					=& |\nabla v|^2 |\overset{\circ}{h}|^2 + 2|\nabla^{T}|\nabla v||^2
					+ \left[ 1+ \frac{1}{2}\left( (p-2)\frac{|\nabla v|^2}{|\nabla v|_{\ve}^2} + 1 \right)^2 \right] v_{\nu\nu}^2.
				\end{aligned}
			\end{equation}
			By using $|\n|\n v||^2=|\n^T|\n v||^2+v_{\nu\nu}^2$ and \eqref{eqn:Kato} in \eqref{xeq1-1},  we see
			\begin{equation}\label{xeq22}
				\begin{aligned}
					{\rm div}(\n|\n v|_{\delta})=& |\n v|^{-1}_{\delta}\Bigg\{|\n v|^2|\overset{\circ}{h}|^2+2|\n^T|\n v||^2+\Big(1+\frac{1}{2}\big( (p-2)\frac{|\nabla v|^2}{|\nabla v|_{\ve}^2} + 1 \big)^2\Big)v_{\nu\nu}^2\\
					&+|\n v|^2{\rm Ric}(\nu,\nu)\Bigg\}
					-\frac{|\n v|^2}{|\n v|^3_{\delta}}(|\n^T|\n v||^2+v_{\nu\nu}^2)+|\n v|^{-1}_{\delta}\langle\n v,\, \n\Delta v\rangle	\\
					\ge&		 |\n v|^{-1}_{\delta}\Bigg\{\frac{1}{2}\frac{|\nabla v|^2}{|\nabla v|_{\ve}^2} \left( (p-2)^2\frac{|\nabla v|^2}{|\nabla v|_{\ve}^2} + 2p-3 \right)v_{\nu\nu}^2+|\n v|^2{\rm Ric}(\nu,\nu)\Bigg\}
					\\&+|\n v|^{-1}_{\delta}\langle\n v,\, \n\Delta v\rangle.		
				\end{aligned}
			\end{equation}
			Here, in the second inequality we have used that $\frac{|\n v|^2}{|\n v|_{\ve}^2}\leq1$ and we also dropped the term involving $|\n^T|\n v||^2$ and $|\overset{\circ}{h}|^2$ which is nonnegative.
			
			Combining \eqref{xeq11} and \eqref{xeq22} and using \eqref{equ-Delta u}, we get the assertion for $|\n v|>0$.

			If $|\n v|=0$ at $x$, since $v_{\nu\nu}$ is bounded as $v$ is smooth, we see that $I_{\ve,\delta}=0$. Hence ${\rm div}V_{\ve,\delta}\leq0$.
		\end{proof}
		
		\begin{Remark}
			We use Hirsch-Miao-Tam's \cite{HMT} idea to prove the above lemma. Here, we extend Hirsch-Miao-Tam's argument from $1<p\leq 2$ (Lemma 3.4 of \cite{HMT}) to $1<p<3$ by using an improved inequality to address $H^{2}$. 
			In \eqref{xeq22}, we used 
			$$
			\begin{aligned}
				H^{2}
				=& \frac{1}{|\n v|^{2}}\left( \frac{(p-1)|\n v|^2+\ve^2}{|\n v|^2_{\ve}} \right)^{2}v_{\nu\nu}^{2} \\
				=& \frac{(p-1)^2}{|\n v|^2}\frac{|\n v|^4}{|\n v|^4_{\ve}}v^2_{\nu\nu} + \left( (2p-3)\frac{|\nabla v|^{2}}{|\nabla v|^2_{\ve}}+1 \right)(1-\frac{|\nabla v|^{2}}{|\nabla v|^2_{\ve}}) \\
				\ge& \frac{(p-1)^2}{|\n v|^2}\frac{|\n v|^4}{|\n v|^4_{\ve}}v^2_{\nu\nu} + (2p-3)\frac{|\nabla v|^{2}}{|\nabla v|^2_{\ve}} (1-\frac{|\nabla v|^{2}}{|\nabla v|^2_{\ve}}), \\
			\end{aligned}
			$$
			while Hirsch-Miao-Tam used $H^2\geq\frac{(p-1)^2}{|\n v|^2}\frac{|\n v|^4}{|\n v|^4_{\ve}}v^2_{\nu\nu}.$ The above estimate can be used to eliminate some unfavorable terms in \eqref{eqn:3.10-0}. 
		\end{Remark}

		\begin{Proposition}\label{lem:3.4}
			Let
			$\{u=f_0(t_1)\}$, $\{u=f_0(t_2)\}$ be two regular level sets for $t_1< t_2$.
			Assume $\alpha_{\ve}(t)\ge 0$ on $(\frac{m_{\ve}}{2},+\infty)$. Then the following inequality holds:
			\begin{eqnarray}\label{asymp-monot}
				F_{\ve}(t_2)-F_{\ve}(t_1)
				\leq \ve\int_{\{f_\ve(t_1)<v<f_\ve(t_2)\}}\frac{|\alpha'_{\ve}(v)-\beta_{\ve}(v)|^2}{\alpha_{\ve}(v)}|\n v|.
			\end{eqnarray}
		\end{Proposition}
		
		\begin{proof}
			
			By Lemma \ref{lem:3.3}, for any $\delta>0$, we have
			\begin{equation*}
				\begin{aligned}
					F_{\ve}(t_2)-F_{\ve}(t_1)=&\int_{\{f_\ve(t_1)<v<f_\ve(t_2)\}}{\rm div}X_{\ve,\delta}\\ \leq&\int_{\{f_\ve(t_1)<v<f_\ve(t_2)\}}{\rm div}U_{\ve}+C\delta+\int_{\{f_\ve(t_1)<v<f_\ve(t_2)\}}({\rm div}W_{\ve,\delta}+I_{\ve,\delta})\mathbf{1}_{\{|\n v|>0\}},
				\end{aligned}
			\end{equation*}
			for some $C > 0$ independent of $\delta$, where $\mathbf{1}_K$ is the characteristic function of $K$.
			It is easy to check that ${\rm div}W_{\ve,\delta}$ and $I_{\ve,\delta}$ is uniformly bounded in $\delta$ for $\{f_\ve(t_1)<v<f_\ve(t_2)\}$. By dominated convergence theorem,
			we have as $\delta\rightarrow0$,
			\begin{equation}\label{eqn:3.8}
				\begin{aligned}
					F_{\ve}(t_2)-F_{\ve}(t_1)	\le &\int_{\{f_\ve(t_1)<v<f_\ve(t_2)\}}({\rm div}U_{\ve}+{\rm div}W_{\ve,0}+I_{\ve,0})\mathbf{1}_{\{|\n v|>0\}}.
				\end{aligned}
			\end{equation}
			Here ${\rm div}W_{\ve,0}$ and $I_{\ve,0}$ are given respectively by ${\rm div}W_{\ve,\d}$ and $I_{\ve,\d}$ for $\d=0$ which make sense.
			On the other hand, by \eqref{equ-H} and Lemma \ref{lem:3.3} (iii), we have that away from  $\{\n v= 0\}$,
			\begin{equation}\label{eqn:3.10-0}
				\begin{aligned}
					I_{\ve,0}=&\alpha_{\ve}(v)|\n v|^{-1}\Bigg\{((2-p)^2-(2-p))\frac{|\n v|^4}{|\n v|_{\ve}^4}v_{\nu\nu}^2-(2-p)\frac{|\n v|^2}{|\n v|_{\ve}^2}v_{\nu\nu}^2\\
					&-\frac{1}{2}\frac{|\nabla v|^2}{|\nabla v|_{\ve}^2} \left( ((p-2)^2+2p-4)\frac{|\nabla v|^2}{|\nabla v|_{\ve}^2} + 2p-3 \right)v_{\nu\nu}^2-|\n v|^2{\rm Ric}(\nu,\nu)\Bigg\}\\
					&+\alpha'_{\ve}(v)\left((2-p)\frac{|\n v|^3}{|\n v|^2_{\ve}}-|\n v|\right)v_{\nu\nu}\\
					=&\alpha_{\ve}(v)|\n v|^{-1}\Bigg\{(2-p)(1-p)\frac{|\n v|^4}{|\n v|_{\ve}^4}v_{\nu\nu}^2-\frac14 H^2|\n v|^2+|\n v|^2K-\frac12|\n v|^2|\overset{\circ} {h}|^2\\
					&-\frac{1}{2}\frac{|\nabla v|^2}{|\nabla v|_{\ve}^2} \left(\left( (p-1)^2-1\right)\frac{|\nabla v|^2}{|\nabla v|_{\ve}^2} +1 \right)v_{\nu\nu}^2-\frac12|\n v|^2R_M\Bigg\}\\
					&+\alpha'_{\ve}(v)\left((2-p)\frac{|\n v|^3}{|\n v|^2_{\ve}}-|\n v|\right)v_{\nu\nu},
				\end{aligned}
			\end{equation} Here in the second inequality  we have used the fact
			that
			$$
			{\rm Ric}(\nu,\nu)=\frac12(R_M-2K_{\S_t}+\frac12H^2-|\overset\circ{h}|^2).$$
			Next, from \eqref{equ-H}, we see that
			\begin{eqnarray}\label{equ-H1}
				H^2\geq\frac{(p-1)^2}{|\n v|^2}\frac{|\n v|^4}{|\n v|^4_{\ve}}v^2_{\nu\nu}.
			\end{eqnarray}
			Using \eqref{equ-H1} and the fact $R_M\ge 0$ in \eqref{eqn:3.10-0}, , we get
			\begin{equation}\label{eqn:3.10}
				\begin{aligned}
					I_{\ve,0}\leq&\alpha_{\ve}(v)\Bigg\{-\frac14(2a+1)(p-1)^2\frac{|\n v|^3}{|\n v|^4_{\ve}}v^2_{\nu\nu}+|\n v|K- \frac12\frac{|\nabla v|}{|\nabla v|_{\ve}^2} \left( 1-\frac{|\nabla v|^2}{|\nabla v|_{\ve}^2} \right)v_{\nu\nu}^2\Bigg\}\\
					&+\alpha'_{\ve}(v)\left\{(1-p)\frac{|\n v|^2}{|\n v|^2_{\ve}}-\left(1-\frac{|\n v|^2}{|\n v|^2_{\ve}}\right)\right\}|\n v|v_{\nu\nu}.
				\end{aligned}
			\end{equation}
			
			From  Lemma \ref{lem:3.3} (i)-(ii) and \eqref{eqn:3.10}, using also  \eqref{eqn:3.5},   we obtain
			\begin{equation}\label{eqn:3.11-1}
				\begin{aligned}
					&{\rm div}U_{\ve}+{\rm div}W_{\ve,0}+I_{\ve,0}	\\
					\leq & 4\pi \gamma'_{\ve}(v) {\rm Cap}_{p,\ve}^{-1}|\n v|_{\ve}^{p-2}|\n v|^2\\
					&+\beta_{\ve}(v)\left\{(3-p)\frac{|\n v|^3}{|\n v|^2_{\ve}}v_{\nu\nu}+\left(1-\frac{|\n v|^2}{|\n v|^2_{\ve}}\right)|\n v|v_{\nu\nu}\right\}+\beta'_{\ve}(v)|\n v|^3
					\\&+\alpha_{\ve}(v)\Bigg\{-\frac14(2a+1)(p-1)^2\frac{|\n v|^3}{|\n v|^4_{\ve}}v^2_{\nu\nu}+|\n v|K- \frac12\frac{|\nabla v|}{|\nabla v|_{\ve}^2} \left( 1-\frac{|\nabla v|^2}{|\nabla v|_{\ve}^2} \right)v_{\nu\nu}^2\Bigg\}\\
					&+\alpha'_{\ve}(v)\left\{(1-p)\frac{|\n v|^2}{|\n v|^2_{\ve}}-\left(1-\frac{|\n v|^2}{|\n v|^2_{\ve}}\right)\right\}|\n v|v_{\nu\nu}
					\\
					=& -4\pi \alpha_{\ve}(v) {\rm Cap}_{p,\ve}^{-1}|\n v|_{\ve}^{p-2}|\n v|^2+\alpha_{\ve}(v)|\n v|K-\alpha_{\ve}(v)\frac14(2a+1)(p-1)^2\frac{|\n v|^3}{|\n v|^4_{\ve}}v^2_{\nu\nu}\\
					&+(2a+1)(1-p)\eta_{\ve}(v)\alpha_{\ve}(v)\frac{|\n v|^3}{|\n v|^2_{\ve}}v_{\nu\nu}-(2 a + 1) (\eta_{\ve}(v))^2  \alpha_{\ve}(v)|\n v|^3\\
					&-(\alpha'_{\ve}(v)-\beta_{\ve}(v))\left(1-\frac{|\n v|^2}{|\n v|^2_{\ve}}\right)|\n v|v_{\nu\nu}-\frac{\alpha_{\ve}(v)}{2}\frac{|\nabla v|}{|\nabla v|_{\ve}^2} \left( 1-\frac{|\nabla v|^2}{|\nabla v|_{\ve}^2} \right)v_{\nu\nu}^2.
				\end{aligned}
			\end{equation}
			Using the Cauchy-Schwarz inequality, we have
			$$-\alpha_{\ve}(v)\frac14(2a+1)(p-1)^2\frac{|\n v|^3}{|\n v|^4_{\ve}}v^2_{\nu\nu}+(2a+1)(1-p)\eta_{\ve}(v)\alpha_{\ve}(v)\frac{|\n v|^3}{|\n v|^2_{\ve}}v_{\nu\nu}-(2 a + 1) (\eta_{\ve}(v))^2  \alpha_{\ve}(v)|\n v|^3\le 0,$$
			$$-(\alpha'_{\ve}(v)-\beta_{\ve}(v))|\n v|v_{\nu\nu}-\frac{\alpha_{\ve}(v)}{2}\frac{|\nabla v|}{|\nabla v|_{\ve}^2} v_{\nu\nu}^2\le \frac{|\alpha'_{\ve}(v)-\beta_{\ve}(v)|^2}{\alpha_{\ve}(v)}|\n v|^2_{\ve}|\n v|.$$
			Therefore,  we deduce from \eqref{eqn:3.11-1} that
			\begin{equation}\label{eqn:3.11}
				\begin{aligned}
					&{\rm div}U_{\ve}+{\rm div}W_{\ve,0}+I_{\ve,0}	\\
					\leq& \ve\frac{|\alpha'_{\ve}(v)-\beta_{\ve}(v)|^2}{\alpha_{\ve}(v)}|\n v|-\alpha_{\ve}(v)\left(4\pi {\rm Cap}_{p,\ve}^{-1}|\n v|_{\ve}^{p-2}|\n v|^2-|\n v|K\right).
				\end{aligned}
			\end{equation}
			Since $v$ is smooth,  by Sard's theorem, the set $\mathcal{A}$ of critical values of $v$ is of measure zero.
			By the co-area formula, using \eqref{eqn:3.11} in \eqref{eqn:3.8},  we have
			\begin{equation*}\label{eqn:3.26}
				\begin{aligned}
					F_{\ve}(t_2)-F_{\ve}(t_1)\leq &\ve\int_{\{f_\ve(t_1)<v<f_\ve(t_2)\}}\frac{|\alpha'_{\ve}(v)-\beta_{\ve}(v)|^2}{\alpha_{\ve}(v)}|\n v|- \int_{f_\ve(t_1)}^{f_\ve(t_2)}\alpha_{\ve}(\tau)
					\left( 4\pi -\int_{\{v=\tau\}} K \right) d\tau\\
					\leq&\ve\int_{\{f_\ve(t_1)<v<f_\ve(t_2)\}}\frac{|\alpha'_{\ve}(v)-\beta_{\ve}(v)|^2}{\alpha_{\ve}(v)}|\n v|.
				\end{aligned}
			\end{equation*}
			where we have used $4\pi -\int_{\{v=\tau\}} K\ge 0$, for each regular level set $\{v=\tau\}$ is connected.
			This completes  the proof of Proposition \ref{lem:3.4}.
		\end{proof}
		
		By letting $\ve\to 0$ in \eqref{asymp-monot}, in view of Lemma \ref{lem:3.2}, we see the following
		\begin{Corollary} Let
			$\{u=f_0(t_1)\}$, $\{u=f_0(t_2)\}$ be two regular level sets for $t_1< t_2$. Then $F(t_2)\le F(t_1)$.
		\end{Corollary}
		This finishes the proof of monotonicity part in  Theorem \ref{thm:Monot}.	
		
		\
		
		\section{Rigidity}
		In this section, we will prove the rigidity part in Theorem \ref{thm:Monot}.
		
		First, we verify $F(t)\equiv 0$ if  $(M, g)=(\mathcal{M}_{m, r_0}^{3},   g_{m})$ is a spatial Schwarzschild manifold outside a rotationally symmetric sphere of mass $m$, given by \eqref{schwarzschild}. Recall that we have shown the following fact in Section 2,
		\begin{eqnarray*}
			&&u(x)= f_{m, r_0}(r)= 1 - \int_{r}^{\infty} \mathfrak{c}_{p} s^{-a-1}(1+\frac{m}{2s})^{-2a}ds
			=1-\frac{I_{a}(\frac{m}{2r})}{I_{a}(\frac{m}{2r_{0}})}, 	\end{eqnarray*}
		where $r=|x|$.
		Let $S_r=\{|x|=r\}=\{u(x)=f_{m, r_0}(r)\}$ be a coordinate sphere. We see from Section 2 that the mean curvature of $S_r$ is given by
		\begin{eqnarray*}
			H=\frac{2}{r}(1 + \frac{m}{2 r})^{-3}\left(1- \frac{m}{2r}\right),	\end{eqnarray*}
		and it satisfies that
		\begin{eqnarray}\label{xeq1}
			\frac{H}{2|\nabla u|_{g}}=\eta(r)= \mathfrak{c}_{p}^{-1}r^{a}\left(1+ \frac{m}{2r}\right)^{2a-1}\left(1- \frac{m}{2r}\right).
		\end{eqnarray}
		By direct computation, we get
		\begin{eqnarray}
			&&\int_{S_r} H^{2} =
			16 \pi \left(1+ \frac{m}{2r}\right)^{-2}\left(1- \frac{m}{2r}\right)^{2},\label{xeq2}\\
			&&\int_{S_r} |\nabla u|^{2} =
			4 \pi \mathfrak{c}_{p}^{2} r^{-2a}\left(1+ \frac{m}{2r}\right)^{-4a}.\label{xeq3}
		\end{eqnarray}
		By the reformation of $F(t)$ in \eqref{equ:F1}, we easily get from  \eqref{xeq1}-\eqref{xeq3} that $F(t)\equiv 0$.

		Second, we show the converse.  Assume that $F(t)$ is a constant on $\mathcal{T}$. By the proof of monotonicity,  we have that along $\S_t$, $t \in \mathcal{T}$,
		\begin{eqnarray}\label{equality-case}
			|\nabla^T|\nabla u|| = 0, \quad R_{M} = 0,\quad |\overset{\circ}{h}| = 0, \quad H = 2 \eta(u)|\nabla u|,
		\end{eqnarray}
		and
		\begin{eqnarray*}
			\int_{\Sigma_{t}} K   = 4 \pi.
		\end{eqnarray*}
		Hence $|\nabla u|$ is a constant on $\Sigma_{t}$, $\Sigma_{t}$ is totally umbilical and it is a topological $2$-sphere.
		Also, we see from \eqref{equality-case} and \eqref{equ: H} that \begin{eqnarray}\label{equality-case-eq2}
			\nabla |\nabla u| = - \frac{1}{p-1} H \nabla u = - (a+1)\eta(u)|\nabla u| \nabla u.
		\end{eqnarray}

		Denote by $$\tilde{\eta}(t) = t^{a+1}\left(1+\frac{m}{2t}\right)^{2a+2}.$$ Abuse of notation, we denote $\tilde{\eta}(u)=\tilde{\eta}(f^{-1}(u))$, where $f^{-1}$ is the inverse function of $f$. We compute that
		\begin{eqnarray*}
			&&\nabla \left( \ln \left( |\nabla u| \tilde{\eta}(u) \right) \right)
			=\frac{\nabla |\nabla u|}{|\nabla u|} + \frac{\tilde{\eta}'(t) \nabla u}{f'(t) \tilde{\eta}(u)}
			= \frac{\nabla |\nabla u|}{|\nabla u|} + (a+1) \eta(u) \nabla u
			= 0.
		\end{eqnarray*}
		In the last equality we used \eqref{equality-case-eq2}.
		Thus $|\nabla u| \tilde{\eta}(u)$ is a constant on $\mathcal{T}$.
		Recall from \eqref{equ:EstimateNablau} that $|\nabla u| = \mathfrak{c}_{p}r^{-a-1} \left( 1+O(r^{-\tilde{\tau}}) \right)$. It follows that
		\begin{eqnarray}\label{equality-case-eq5}
			|\nabla u| =\frac{\mathfrak{c}_{p}}{\tilde{\eta}(t)}=\mathfrak{c}_{p} t^{-a-1}\left(1+\frac{m}{2t}\right)^{-2a-2} \hbox{ along }\S_t=\{u=f(t)\}.
		\end{eqnarray}
		Therefore,  up to isometry, $M = [r_{0},+ \infty) \times \p M$ 
		with its metric $g$ given by
		\begin{eqnarray}\label{metric-form}
			g = \frac{(\tilde{\eta}(t)f'(t))^2}{\mathfrak{c}_{p}^2} dt \otimes dt
			+ g_{\alpha \beta}(t, \vartheta) d \vartheta^{\alpha} \otimes d \vartheta^{\beta},
		\end{eqnarray}
		where $\{ \vartheta^{\alpha} \}$ is local chart of $\p M$ and $g_{\alpha \beta}(t, \vartheta) d \vartheta^{\alpha} \otimes d \vartheta^{\beta}$ represents the metric on $\S_t$  induced by $g$.
		From \eqref{metric-form}, we see that the second fundamental form of $\S_t$ is given by
		\begin{eqnarray}\label{equality-case-eq3}
			h_{\alpha \beta}
			= - \frac{1}{2}|\nabla u| \frac{\partial}{\partial u}g_{\alpha \beta}.
		\end{eqnarray}
		Note from \eqref{equality-case} that \begin{eqnarray}\label{equality-case-eq4}
			h_{\alpha \beta} = \frac{H}{2} g_{\alpha \beta}=\eta(u)|\n u|g_{\alpha \beta}.
		\end{eqnarray} It follows from \eqref{equality-case-eq3} and \eqref{equality-case-eq4} that	$$
		\frac{\partial g_{\alpha \beta}}{\partial u} = - 2 \eta (u) g_{\alpha \beta}.
		$$
		In other words,
		\[
		\frac{\partial g_{\alpha \beta}}{\partial t}
		= - 2 \eta (t) f'(t) g_{\alpha \beta}
		= - 2 t^{-1}\left(1+\frac{m}{2t}\right)^{-1}\left(1-\frac{m}{2t}\right) g_{\alpha \beta}.
		\]
		It follows that
		\begin{eqnarray}\label{equality-case-eq12}
			g_{\alpha \beta}(t, \vartheta) = t^{2}\left(1+\frac{m}{2t}\right)^{4} c_{\alpha \beta} (\vartheta).
		\end{eqnarray}
		for some metric $(c_{\alpha \beta})$ on $\p M$.
		Next, we determine $(c_{\alpha \beta})$ to be the round metric on $\mathbb{S}^2$. To achieve this,
		we compute the Gauss curvature $K_{\S_t}$ of $\S_t$.
		The Gauss equation tells that\begin{eqnarray}\label{gauss-eq}
			2 K_{\S_t} = R_{M} - 2 \operatorname{Ric}_M(\nu, \nu) + H^2 - |h|^2.
		\end{eqnarray}
		We have already known that along $\S_t$, \begin{eqnarray}
			&&H = 2 \eta(u)|\nabla u| = 2 t^{-1}\left(1+\frac{m}{2t}\right)^{-3}\left(1-\frac{m}{2t}\right), \label{equality-case-eq8}\\
			&&|h|^2 = \frac{H^2}{2} = 2 t^{-2}\left(1+\frac{m}{2t}\right)^{-6}\left(1-\frac{m}{2t}\right)^2, \label{equality-case-eq9}\\
			&&R_{M} = 0. \label{equality-case-eq10}
		\end{eqnarray}
		So we only need to calculate $\operatorname{Ric}_M(\nu, \nu)$.
		Using the evolution equation  \eqref{equ: dH_t} of $H$, we have
		\begin{eqnarray}\label{equality-case-eq6}
			\frac{\p}{\p t} H
			&=&- f'(t) \left( \Delta_{\Sigma_t} \left( \frac{1}{|\nabla u|} \right)
			+ \left\{ |h|^2 + \operatorname{Ric}_M (\nu, \nu) \right)\frac{1}{|\nabla u|} \right\} \nonumber
			\\&=& - \left(1+\frac{m}{2t}\right)^2 \left( |h|^2 + \operatorname{Ric}_M (\nu, \nu) \right),
		\end{eqnarray}
		where we use the fact that $|\nabla u|$ is a constant given by \eqref{equality-case-eq5} along $\S_t$.
		On the other hand,
		\begin{eqnarray}\label{equality-case-eq7}
			\frac{\p}{\p t} H&=& \frac{d}{d t} \left(  2 t^{-1}(1+\frac{m}{2t})^{-3}\left(1+\frac{m}{2t}\right)  \right) \nonumber
			\\&=& - 2 t^{-2}\left(1+\frac{m}{2t}\right)^{-4} \left( 1-4\frac{m}{2t}+(\frac{m}{2t})^2 \right).
		\end{eqnarray}
		From \eqref{equality-case-eq6} and \eqref{equality-case-eq7}, we get\begin{eqnarray}\label{equality-case-eq11}
			\operatorname{Ric}_M (\nu, \nu) = - 4 t^{-2}\left(1+\frac{m}{2t}\right)^{-6}\frac{m}{2t}.
		\end{eqnarray}
		Finally, it follows from \eqref{gauss-eq}, \eqref{equality-case-eq8}, \eqref{equality-case-eq9}, \eqref{equality-case-eq10} and \eqref{equality-case-eq11} that
		\begin{eqnarray}\label{equality-case-eq13}
			K_{\S_t} = t^{-2}\left(1+\frac{m}{2t}\right)^{-4}.
		\end{eqnarray}
		In particular, $\partial M$ has constant positive Gauss curvature. Moreover,
		in view of \eqref{equality-case-eq12} and \eqref{equality-case-eq13}, we see that  $(c_{\alpha \beta})$ is the round metric on $\mathbb{S}^2$.
		We conclude that
		$$g=\left(1+\frac{m}{2t}\right)^4 \left(dt \otimes dt
		+ t^{2}g_{\mathbb{S}^{2}}\right)
		$$
		which is exactly the Schwarzschild metric of mass $m$. It is clear from \eqref{equality-case-eq4} that $\p M$ is a rotationally symmetric sphere.
		
		We finish the proof of rigidity part in Theorem \ref{thm:Monot} and thus complete the proof of Theorem \ref{thm:Monot}.
		
		\
		
		\section{Asymptotic behavior}
		Recall that
		\begin{equation*}
			F(t)= 4 \pi \gamma(t) + \alpha(t) \int_{\Sigma_t} H |\nabla u| + \beta(t) \int_{\Sigma_t} |\nabla u|^2,
		\end{equation*}
		where 	\begin{equation*}
			\begin{aligned}
				\alpha(t)
				=& t\left(1+\frac{m}{2t}\right)^{2}
				\left\{ \left( C_{2}\mathfrak{c}_{p}+C_{1}\frac{a}{m} \frac{I_{a}(\frac{m}{2t})}{I_{a}(\frac{m}{2r_0})}  \right)
				\eta(t)-C_{1}\frac{1}{m} \right\},
			\end{aligned}
		\end{equation*}
		\begin{equation*}
			\begin{aligned}
				\beta(t)
				=&- \eta(t)\alpha(t)
				+ \left( C_{2}\mathfrak{c}_{p}\frac{m}{a}+C_{1} \frac{I_{a}(\frac{m}{2t})}{I_{a}(\frac{m}{2r_0})} \right) \mathfrak{c}_{p}^{-2} t^{2a}\left( 1+\frac{m}{2t}\right)^{4a},
			\end{aligned}
		\end{equation*}
		\begin{equation*}
			\begin{aligned}
				\gamma(t)
				=&- \mathfrak{c}_{p}^{2} t^{-2a}\left( 1+\frac{m}{2t}\right)^{-4a} \eta(t) \alpha(t)
				-  \left( C_{2}\mathfrak{c}_{p}\frac{m}{a}+C_{1} \frac{I_{a}(\frac{m}{2t})}{I_{a}(\frac{m}{2r_0})} \right).
			\end{aligned}
		\end{equation*}
		Notice that $$\eta(t)=\mathfrak{c}_{p}^{-1}t^{a}\left(1 + \frac{m}{2 t}\right)^{2a-1}\left(1- \frac{m}{2t}\right)>0$$ as long as $m\neq0$ and $t>r_{0}\ge \frac{|m|}{2}$. We may rewrite $F(t)$ as follows:
		\begin{eqnarray}\label{equ:F1}
			F(t)
			&=
			& - \left( C_{2}\mathfrak{c}_{p}\frac{m}{a}+C_{1}\frac{I_{a}(\frac{m}{2t})}{I_{a}(\frac{m}{2r_{0}})} \right)
			\left(
			4 \pi - \mathfrak{c}_{p}^{-2} t^{2a}(1+\frac{m}{2t})^{4a} \int_{\Sigma_{t}} |\nabla u|^2
			\right) \\
			&& - \frac{\alpha(t)}{4\eta(t)} \int_{\Sigma_{t}} \left( H-2\eta(t)|\nabla u| \right) ^2
			- \frac{1}{4} \frac{\alpha(t)}{\eta(t)}
			\left( 16\pi - \int_{\Sigma_{t}} H^2 \right) \nonumber
			\\&&-4\pi \frac{\alpha(t)}{\eta(t)}\left( \frac{(1-\frac{m}{2t})^{2}}{(1+\frac{m}{2t})^{2}}-1 \right). \nonumber
		\end{eqnarray}
		To estimate the limit of $F(t)$ as $t\rightarrow+\infty$, we need to know the limit of some relevant quantities.
		\begin{Lemma}
			\begin{eqnarray}
				&&\lim_{t\to\infty} \left( \frac{\alpha(t)}{\eta(t)}-C_{2}\mathfrak{c}_{p}t(1 + \frac{m}{2t})^{2} \right) =0, \label{equ:EstimateAlphaDevidedByEta}\\
				&&\lim_{t\to\infty} 4\pi \frac{\alpha(t)}{\eta(t)}\left( \frac{(1-\frac{m}{2t})^{2}}{(1+\frac{m}{2t})^{2}}-1 \right)
				= -8 \pi C_{2}\mathfrak{c}_{p}m.\label{xeq-limit6}
			\end{eqnarray}
		\end{Lemma}
		\begin{proof}
			We see easily that as $t\rightarrow+\infty$,
			\begin{eqnarray}
				&&I_{a}(\frac{m}{2t}) \rightarrow I_{a}(0)=0,\nonumber\\
				&& (\eta(t))^{-1}= \mathfrak{c}_{p} t^{-a} \left(1 + \frac{m}{2t}\right)^{-2a+1} \left(1- \frac{m}{2t}\right)^{-1}=\mathfrak{c}_{p} t^{-a} (1+O(t^{-1}))
				\rightarrow 0.\label{eta-limit}
			\end{eqnarray}
			Consider the function
			$$
			\begin{aligned}
				\xi(t)
				=& \frac{a}{m} \frac{I_{a}(\frac{m}{2t})}{I_{a}(\frac{m}{2r_{0}})}-\frac{1}{m} (\eta(t))^{-1} \\
				=& \frac{a}{m}\mathfrak{c}_{p} \int_{t}^{+\infty}s^{-a-1}\left(1+\frac{m}{2s}\right)^{-2a}ds
				-\frac{1}{m} \mathfrak{c}_{p} t^{-a} \left(1 + \frac{m}{2t}\right)^{-2a+1} \left(1- \frac{m}{2t}\right)^{-1} \\
				=& \mathfrak{c}_{p} \int_{t}^{+\infty}s^{-a-2}\left(1+\frac{m}{2s}\right)^{-2a}\left(1-\frac{m}{2s}\right)^{-2}ds.
			\end{aligned}
			$$
			When $t\rightarrow+\infty$, $\lim\limits_{t\rightarrow+\infty}\xi(t)=0$,
			by L'Hospital formula,
			\begin{equation}\label{equ:LimitofXi}
				\lim\limits_{t\rightarrow+\infty} \frac{\xi(t)}{t^{-a-1}}=\mathfrak{c}_{p}\frac{1}{a+1}.
			\end{equation}
			Notice that
			$$
			\frac{\alpha(t)}{\eta(t)}
			=t\left(1+\frac{m}{2t}\right))^{2} \left( C_{2}\mathfrak{c}_{p}+C_{1}\xi(t) \right),
			$$
			then
			$$
			\lim\limits_{t\rightarrow+\infty} \left( \frac{\alpha(t}{\eta(t)}-C_{2}\mathfrak{c}_{p}t(1 + \frac{m}{2t})^{2} \right) = \lim\limits_{t\rightarrow+\infty} t(1+\frac{m}{2t})^{2} C_{1}\xi(t) = 0.
			$$
			Since
			\begin{eqnarray*}
				\frac{\left(1-\frac{m}{2t}\right))^{2}}{\left(1+\frac{m}{2t}\right))^{2}}-1=-\frac{2m}{t}(1+o(1)), \hbox{ as }t\to\infty,
			\end{eqnarray*}
			we get the second assertion.
		\end{proof}
		
		Similar to \cite[Lemma 2.1 and Lemma 2.2]{M1}, we have the following.
		\begin{Lemma}
			Assume that $f(t)$ is a regular value of $u$. Let $0 < \tilde{\tau} < \min \{ \tau,1 \}$. Then, along $\Sigma_{t} = \{ u = f(t) \}$, we have  that as $t \rightarrow + \infty$,
			\begin{eqnarray}
				&&|\Sigma_{t}| = 4 \pi t^{2} \left( 1 + O(t^{- \tilde{\tau}}) \right),\label{equ:Area}
				\\&&\int_{\Sigma_t} H |\nabla u| = 8 \pi \mathfrak{c}_{p} t^{- a} \left( 1 + O(t^{-  \tilde{\tau}}) \right),\label{equ:IntegrateHNablau}	
				\\&&\int_{\Sigma_t} |\nabla u|^{2} = 4 \pi \mathfrak{c}_{p}^{2} t^{- 2 a} \left( 1 + O(t^{-  \tilde{\tau}}) \right). \label{equ:IntegrateNablauSquare}
			\end{eqnarray}
		\end{Lemma}
		\begin{proof}
			Let $r=|x|$. we know from \eqref{equ:asymptotic u} that
			\begin{equation}\label{xeq-limit1}
				u=1 - \frac{\mathfrak{c}_{p}}{a} r^{-a} \left( 1 + O_{2}(r^{- \tilde{\tau}}) \right), \hbox{ as }r\to \infty.
			\end{equation}
			for any $0 < \tilde{\tau} < \min \{ \tau,1 \}$. On the other hand,
			\begin{eqnarray}\label{xeq-limit2}
				f(t) = 1 - \int_{t}^{\infty} \mathfrak{c}_{p} s^{-a-1}(1+\frac{m}{2s})^{-2a} ds
				= 1 - \frac{\mathfrak{c}_{p}}{a} t^{-a} \left( 1 + O(t^{- 1}) \right),\hbox{ as }t\to \infty.
			\end{eqnarray}
			Since $u = f(t)$ on $\Sigma_{t}$, we see that $u \rightarrow 1$ and $r \rightarrow \infty$ as $t \rightarrow \infty$. Moreover, it follows from \eqref{xeq-limit1} and \eqref{xeq-limit2} that
			\begin{eqnarray}\label{xeq-limit3}
				r = t \left( 1 + O(t^{- \tilde{\tau}}) \right).
			\end{eqnarray}

			Furthermore, we derive from \eqref{xeq-limit1} that
			\begin{eqnarray}\label{equ:EstimateNablau}
				&&\nabla u = {\mathfrak{c}_{p}} r^{-a-1} \nabla r \left( 1 + O_{1}(r^{- \tilde{\tau}}) \right), \nonumber
				\\&&	|\nabla u|_{g} = {\mathfrak{c}_{p}} r^{-a-1} \left( 1 + O_{1}(r^{- \tilde{\tau}}) \right),
				\\&&(\n^2 u)_{ij} = {\mathfrak{c}_{p}} r^{-a-2} \left[ -(a + 2) \frac{x_i}{r} \frac{x_j}{r} + \delta_{i j} \right] \left( 1 + O(r^{- \tilde{\tau}}) \right).\nonumber
			\end{eqnarray}
			By \eqref{equ: H}, we have
			\begin{equation} \label{equ:EstimateH}
				H
				= - (p - 1) |\nabla u|^{- 1} \n^2 u \left(\frac{\nabla u}{|\nabla u|_{g}}, \frac{\nabla u}{|\nabla u|_{g}}\right)
				= \frac{2}{r} \left( 1 + O(r^{- \tilde{\tau}}) \right).
			\end{equation}
			Hence, as $t$ large enough, $H>0$ along $\S_t$. It follows that $\Sigma_{t}$ is area outer-minimizing in $(M, g)$ when $t$ is large enough, since the exterior of $\Sigma_{t}$ in $M$ is foliated by mean-convex surfaces. Similarly, each coordinate sphere $S_r$ is area outer-minimizing in $(M, g)$ when $r$ is large enough.
			Denote $r_{-}(t) = \min\limits_{\Sigma_t} |x|$ and $r_{+}(t) = \max\limits_{\Sigma_t} |x|$. It follows from the outer-minimizing property that \begin{eqnarray}\label{xeq-limit4}
				|S_{r_{-}(t)}| \le |\Sigma_t| \le |S_{r_{+}(t)}|.
			\end{eqnarray}
			By the asymptotical flatness \eqref{af}, we have \begin{eqnarray}\label{xeq-limit5}
				|S_r|=4\pi r^2\left( 1 + O(r^{- {\tau}}) \right).
			\end{eqnarray}
			The first assertion \eqref{equ:Area} follow directly from \eqref{xeq-limit3}, \eqref{xeq-limit4} and \eqref{xeq-limit5}. Finally, the assertions \eqref{equ:IntegrateHNablau} and \eqref{equ:IntegrateNablauSquare} follow  from \eqref{equ:Area}, \eqref{equ:EstimateNablau} and \eqref{equ:EstimateH}
		\end{proof}

		\begin{Lemma}
			\begin{equation} \label{equ:LimitofF1}
				\lim\limits_{t\rightarrow+\infty}
				\left( 4\pi-\mathfrak{c}_{p}^{-2} t^{2a}(1+\frac{m}{2t})^{4a} \int_{\Sigma_{t}} |\nabla u|^2   \right)= 0.
			\end{equation}
			
		\end{Lemma}
		\begin{proof}It follows directly by \eqref{equ:IntegrateNablauSquare}.
		\end{proof}
		
		\begin{Lemma}
			\begin{equation} \label{equ:LimitofF2}
				\lim_{t \rightarrow \infty} \frac{\alpha(t)}{\eta(t)} \int_{\Sigma_{t}} \left( \frac{H}{2} - \eta(u) |\nabla u| \right)^{2} = 0.
			\end{equation}
		\end{Lemma}
		
		\begin{proof}
			We see from \eqref{eta-limit}, \eqref{equ:EstimateNablau}, \eqref{equ:EstimateH}, \eqref{xeq-limit3} that for $\tilde{\tau}<1$,
			\begin{eqnarray*}
				&&\frac{H}{2} - \eta(u) |\nabla u|
				=\frac{1}{t} \left( 1 + O(t^{- \tilde{\tau}}) \right)
				- \mathfrak{c}_{p}^{-1} t^{a}(1+O(t^{-1}))
				{\mathfrak{c}_{p}} t^{-a-1} \left( 1 + O_{1}(t^{- \tilde{\tau}}) \right)
				=O(t^{-1-\tilde{\tau}}).
			\end{eqnarray*}
			Combining with \eqref{equ:Area} and \eqref{equ:EstimateAlphaDevidedByEta}, we have that
			\begin{eqnarray*}
				&&\frac{\alpha(t)}{\eta(t)} \int_{\Sigma_{t}} \left( \frac{H}{2} - \eta(u) |\nabla u| \right)^{2}
				= O(t^{1 - 2\tilde{\tau}}).
			\end{eqnarray*}
			Since $\tau > \frac{1}{2}$, we may take $\frac12<\tilde{\tau}< \min\{ \tau,1 \}$ to get the assertion.
		\end{proof}
		
		\begin{Lemma}Assume $C_2\ge 0$. Then
			\begin{equation} \label{equ:LimitofF3}
				\lim_{t \rightarrow \infty} \frac14\frac{\alpha(t)}{\eta(t)}\left( 16 \pi - \int_{ \S_t} H^{2}   \right) \le 8 \pi C_{2}\mathfrak{c}_{p} \mathfrak{m}_{ADM}.
			\end{equation}
			Furthermore, when $C_{2}=0$,
			$$
			\lim\limits_{t \rightarrow \infty} \frac14\frac{\alpha(t)}{\eta(t)}\left( 16 \pi - \int_{  \S_t} H^{2} \right) =0.
			$$
		\end{Lemma}
		
		\begin{proof}
			In \cite[Lemma 2.5]{AMMO}, it was proved that
			\begin{eqnarray}\label{limit-ammo}
				\lim_{s \rightarrow \infty} \frac{s}{4}\left( 16 \pi - \int_{ \left\{ u=1-\frac{\mathfrak{c}_{p}}{a}s^{-a} \right\} } H^{2} \right) \le 8 \pi \mathfrak{m}_{ADM}.
			\end{eqnarray}
			Consider the change of variable $f(t)=1-\frac{\mathfrak{c}_{p}}{a}s^{-a}$. Then by \eqref{xeq-limit2}, we see
			\begin{eqnarray}\label{xeq-limit7}
				s=t \left( 1+O(t^{-1}) \right), \quad t\to\infty.
			\end{eqnarray}
			Thus the assertion follows from \eqref{limit-ammo}, \eqref{equ:EstimateAlphaDevidedByEta} and \eqref{xeq-limit7}.
			When $C_{2}=0$, recalling \eqref{equ:Area} and \eqref{equ:EstimateH}, we get
			$$
			\lim\limits_{t \rightarrow \infty} \left( 16 \pi - \int_{ \S_t} H^{2}\right) =0.
			$$
			Then the second assertion is proved by using \eqref{equ:EstimateAlphaDevidedByEta}.
		\end{proof}
		
		Combining  \eqref{xeq-limit6}, \eqref{equ:LimitofF1},  \eqref{equ:LimitofF2} and  \eqref{equ:LimitofF3}, we get the limit of $F(t)$.
		\begin{Proposition}\label{limit-F}
			Assume $C_2\ge 0$. Then
			\begin{equation} \label{equ:LimitofF}
				\begin{aligned}
					\lim\limits_{t\rightarrow+\infty} F(t)
					\ge -8 \pi C_{2}\mathfrak{c}_{p} (\mathfrak{m}_{ADM}-m).
				\end{aligned}
			\end{equation}
			Furthermore, when $C_{2}=0$,
			$$
			\lim\limits_{t\rightarrow+\infty} F(t)=0.
			$$
		\end{Proposition}
		

		
		\
		
		\section{Applications and proof of main results}
		Combining Theorem \ref{thm:Monot} and Proposition \ref{limit-F}, we get the following
		\begin{Corollary} Let $(M, g)$ and $u$ be as in Theorem \ref{Thm:1.02}. For any $k\in(-1,0)\cup(0,1]$,
			let $m=2 \operatorname{sgn}(k) \left(I_a(k)\mathfrak{c}_{p}\right)^{\frac{1}{a}}$, $r_{0}=\frac{m}{2k}=|k|^{-1}\left(I_a(k)\mathfrak{c}_{p}\right)^{\frac{1}{a}}$. Let $\a, \b, \g$ be three one-variable functions given by \eqref{ode-solution1} with  $$
			C_{2} \ge 0, \hbox{ and }
			\left( C_{2}\mathfrak{c}_{p}+C_{1}\frac{a}{m} \right)
			\eta(r_{0}) \ge C_{1}\frac{1}{m}. 	$$	Then	we have that
			\begin{equation*}
				\begin{aligned}
					F(t)=4 \pi \gamma(t) + \alpha(t) \int_{\Sigma_{t}} H |\nabla u|   + \beta(t) \int_{\Sigma_{t}} |\nabla u|^2
					\ge  -C_{2}\mathfrak{c}_{p}8\pi(\mathfrak{m}_{ADM}-m).
				\end{aligned}
			\end{equation*}
			Moreover, equality  holds if and only if $(M, g)$ is isometric to  the spatial Schwarzschild manifold of mass $m$ outside a rotationally symmetric ball, $(\mathcal{M}_{m, r_0}^{3},  g_{m})$.	
		\end{Corollary}
		
		Next we consider the following cases, including $k=1$, $0<k<1$ and $-1<k<0$.
		
		\noindent{\bf Case 1: $k=1$.}

		In this case, $m$ and $r_{0}$ satisfies that $m=2 \left(I_a(1)\mathfrak{c}_{p}\right)^{\frac{1}{a}}$ and $r_0=\frac{m}{2}$. We see that
		\begin{eqnarray*}
			&&\eta(\frac{m}{2})=0,\quad \alpha(\frac{m}{2}) =-2C_{1},\\
			&&\beta(\frac{m}{2})
			= \left( C_{2}\mathfrak{c}_{p}\frac{m}{a}+C_{1} \right)
			2^{4a} (I_{a}(1))^{2},\\
			&&\gamma(\frac{m}{2})
			= - \left( C_{2}\mathfrak{c}_{p}\frac{m}{a}+C_{1} \right) .
		\end{eqnarray*}
		Hence we have the following consequence.
		\begin{Corollary}\label{prop:6.2}
			Let $(M, g)$ and $u$ be as in Theorem \ref{Thm:1.02}. Let $m=2 \left(I_a(1)\mathfrak{c}_{p}\right)^{\frac{1}{a}}$. Let $\a, \b, \g$ be three one-variable functions given by \eqref{ode-solution1} with $$C_2\ge 0,\quad C_1\le 0.$$
			Then we have that
			\begin{equation}\label{ineq-cor}
				\begin{aligned}
					&- 4 \pi \left( C_{2}\mathfrak{c}_{p}\frac{m}{a}+C_{1} \right)
					- 2C_{1} \int_{\Sigma} H |\nabla u|
					+ \left( C_{2}\mathfrak{c}_{p}\frac{m}{a}+C_{1} \right) 2^{4a} (I_{a}(1))^{2} \int_{\Sigma} |\nabla u|^2 \\
					\ge & -C_{2}\mathfrak{c}_{p}8\pi(\mathfrak{m}_{ADM}-m).
				\end{aligned}
			\end{equation}
			Moreover, equality sign holds if and only if $(M, g)$ is isometric to  the spatial Schwarzschild manifold of mass $m$,   $(\mathcal{M}_{m}^{3},  g_{m})$.
		\end{Corollary}
		
		\
		
		\noindent{\it Proof of Theorem \ref{Thm:1.01}.}
		
		Let $C_{2}=0,C_{1}=-1$ in \eqref{ineq-cor},  we obtain \eqref{geom-ineq-1}. Let $C_{2}=\mathfrak{c}_{p}^{-1}$ and $C_{1}=0$ in \eqref{ineq-cor}, we obtain \eqref{geom-ineq-2}. \qed
		
		
		\
		
		\noindent{\it Proof of Theorem \ref{Thm:1.03}.}
		Since $H=0$ on $\S$, we see from \eqref{geom-ineq-1} and \eqref{geom-ineq-2} that
		\begin{equation*}
			\begin{aligned}
				0\le
				4 \pi
				- 2^{4a} (I_{a}(1))^{2} \int_{\Sigma} |\nabla u|^2
				\le 8\pi a \left( \frac{\mathfrak{m}_{ADM}}{m}-1 \right) .
			\end{aligned}
		\end{equation*}
		This is \eqref{geom-ineq-4}. On the other hand,
		By H\"older's inequality,
		\begin{equation*}
			\begin{aligned}
				4\pi \mathfrak{c}_{p}^{p-1}
				= \int_{\Sigma}|\nabla u|^{p-1}
				\le \left( \int_{\Sigma}|\nabla u|^{2} \right)^{\frac{p-1}{2}} |\Sigma|^{\frac{3-p}{2}}
				\le \left( \frac{4 \pi}{2^{4a} (I_{a}(1))^{2}} \right)^{\frac{p-1}{2}} |\Sigma|^{\frac{3-p}{2}} .
			\end{aligned}
		\end{equation*}
		It follows that
		\begin{equation*}
			\sqrt{\frac{|\Sigma|}{16\pi}} \ge 2 \left( I_{a}(1) \mathfrak{c}_{p} \right)^{\frac{1}{a}}.
		\end{equation*}
		This is \eqref{geom-ineq-5}.
		\qed
		
		\
		
		\noindent{\bf Case 2: $-1<k<0$ and $0<k<1$.}
		
		In this case, $m$ and $r_{0}$ satisfies that $m=2 \operatorname{sgn}(k) \left(I_a(k)\mathfrak{c}_{p}\right)^{\frac{1}{a}}$ and $r_0=\frac{m}{2k}$. We see that
		\begin{eqnarray*}
			&&\eta(r_{0})=I_{a}(k)|k|^{-a}(1+k)^{2a-1}(1-k), \\
			&&\alpha(r_{0})
			=m\frac{(1+k)^{2}}{2k}
			\left\{
			\left( C_{2}\mathfrak{c}_{p}+C_{1}\frac{a}{m} \right) \eta(r_{0}) -C_{1}\frac{1}{m}
			\right\}
			,\\
			&&\beta(r_{0})
			= -\eta(r_{0})\alpha(r_{0})
			+ \left( C_{2}\mathfrak{c}_{p}\frac{m}{a}+C_{1} \right) (I_{a}(k))^{2}|k|^{-2a}(1+k)^{4a} ,\\
			&&\gamma(r_{0})
			= -(I_{a}(k))^{-2}|k|^{2a}(1+k)^{-4a}\eta(r_0)\alpha(r_{0})
			- \left( C_{2}\mathfrak{c}_{p}\frac{m}{a}+C_{1} \right)  .
		\end{eqnarray*}
		Then we have the following consequence.
		\begin{Corollary}\label{prop:6.3}
			Let $(M, g)$ and $u$ be as in Theorem \ref{Thm:1.02}. For any $k\in(-1,0)\cup(0,1)$,
			let $m=2 \operatorname{sgn}(k) \left(I_a(k)\mathfrak{c}_{p}\right)^{\frac{1}{a}}$, $r_{0}=\frac{m}{2k}=|k|^{-1}\left(I_a(k)\mathfrak{c}_{p}\right)^{\frac{1}{a}}$. Let $\a, \b, \g$ be three one-variable functions given by \eqref{ode-solution1} with
			$$
			C_{2} \ge 0, \hbox{ and }
			\left( C_{2}\mathfrak{c}_{p}+C_{1}\frac{a}{m} \right)
			\eta(r_{0}) \ge C_{1}\frac{1}{m}. 			
			$$
			Then we have that
			\begin{equation}\label{ineq-cor2}
				\begin{aligned}
					&- 4 \pi \left\{ \frac{(1-k)^{2}}{(1+k)^{2}\eta(r_{0})}\alpha(r_{0})
					+ \left( C_{2}\mathfrak{c}_{p}\frac{m}{a}+C_{1} \right) \right\}
					+ \alpha(r_{0}) \int_{\Sigma} H |\nabla u|  \\
					&+ \left( -\eta(r_{0})\alpha(r_{0})
					+ \left( C_{2}\mathfrak{c}_{p}\frac{m}{a}+C_{1} \right) \frac{(1+k)^{2}(\eta(r_{0}))^{2}}{(1-k)^{2}} \right)  \int_{\Sigma} |\nabla u|^2  \\
					\ge & -C_{2}\mathfrak{c}_{p}8\pi(\mathfrak{m}_{ADM}-m).
				\end{aligned}
			\end{equation}
			Moreover, equality sign holds if and only if $(M, g)$ is isometric to  the spatial Schwarzschild manifold of mass $m$ outside a rotationally symmetric ball, $(\mathcal{M}_{m, r_0}^{3},  g_{m})$.
		\end{Corollary}
		
		\
		
		\noindent{\it Proof of Theorem \ref{Thm:1.02}.}
		Notice that for any $k \in (-1,0)\cup(0,1)$, the inequality $\frac{1-a\eta(r_{0})}{m}>0$ holds.  
		Let $C_{2}=0$ and $C_{1}=-1$ in \eqref{ineq-cor2},  we obtain \eqref{geom-ineq-10}. Let $C_{2}=\mathfrak{c}_{p}^{-1}$ and $C_{1}=\frac{m\eta(r_{0})}{1-a\eta(r_{0})}$ in \eqref{ineq-cor}, we see that $$\left( C_{2}\mathfrak{c}_{p}+C_{1}\frac{a}{m} \right)
		\eta(r_{0})-C_{1}\frac{1}{m}=0.$$ It follows that $\alpha(r_0)=0$. Then we easily obtain \eqref{geom-ineq-20}. \qed
		
		\
		
		\noindent{\it Proof of Theorem \ref{Thm:1.04}.} 	Let $k\in(-1,1]$ be such that
		\begin{eqnarray}\label{k}
			1 - \frac{1}{16\pi} \int_{\S} H^{2} = \frac{4k}{(1+k)^{2}}.
		\end{eqnarray}
		
		\noindent{\bf Case 1: $k\in (-1,0)\cup(0,1]$.}
		We see from \eqref{geom-ineq-10} and \eqref{geom-ineq-20} that
		\begin{equation}\label{xeq-100}
			\begin{aligned}
				& \frac{(1+k)^{2}r_{0}}{\eta(r_{0})}
				\left\{
				4 \pi \frac{(1-k)^{2}}{(1+k)^{2}} - \frac{1}{4} \int_{\Sigma} H^{2}
				+ \int_{\Sigma} \left( \frac{H}{2}-\eta(r_{0})|\nabla u| \right) ^{2}
				\right\}
				\le 8\pi\left( \mathfrak{m}_{ADM}-m \right).
			\end{aligned}
		\end{equation}
		Substituting  $k$ given in \eqref{k} into \eqref{xeq-100}, we have
		\begin{equation*}
			\begin{aligned}
				&0\le \frac{(1+k)^{2}r_{0}}{\eta(r_{0})} \int_{\Sigma} \left( \frac{H}{2}-\eta(r_{0})|\nabla u| \right) ^{2}
				\le 8\pi\left( \mathfrak{m}_{ADM}-m \right). 
			\end{aligned}
		\end{equation*}
		This gives \eqref{geom-ineq-4}.
		
		It follow from \eqref{geom-ineq-10} that
		\begin{equation*}
			4 \pi \geq\frac{(1+k)^{2}}{(1-k)^{2}} (\eta(r_{0}))^{2} \int_{\Sigma} |\nabla u|^2.
		\end{equation*}
		Then, by H\"older's inequality,
		\begin{equation*}
			\begin{aligned}
				4\pi \mathfrak{c}_{p}^{p-1}
				= \int_{\Sigma}|\nabla u|^{p-1}
				\le \left( \int_{\Sigma}|\nabla u|^{2} \right)^{\frac{p-1}{2}} |\Sigma|^{\frac{3-p}{2}}
				\le \left( \frac{4 \pi}{\mathfrak{c}_{p}^{-2}r_{0}^{2a}(1+k)^{4a}} \right)^{\frac{p-1}{2}} |\Sigma|^{\frac{3-p}{2}} .
			\end{aligned}
		\end{equation*}
		
		Thus
		$$
		\sqrt{\frac{|\Sigma|}{16\pi}}\geq(1+k)^{2}r_{0}=  \frac{(1+k)^2}{2|k|}(I_a(k)\mathfrak{c}_p)^{\frac1a}.
		$$
		The equality holds if and only if $M$ is isometric to $(\mathcal{M}_{m}^{3},  g_{m})$.
		
		\noindent{\bf Case 2: $k=0$.} In this case, by using L'Hospital formula, let $k\rightarrow0$, we can see that $m\rightarrow0$,  $r_{0}\rightarrow(\frac{\mathfrak{c}_p}{a})^{\frac1a}$ and $\frac{1-a\eta(r_{0})}{2k}\rightarrow\frac{a^{2}}{a+1}$. 
		Also, in view of Remark \ref{rmk1.3} (iii),  we may recover the procedure along this paper and we get
		$$
		\mathfrak{m}_{ADM}\ge0, \quad
		\sqrt{\frac{|\Sigma|}{16\pi}}\geq\frac{1}{2}(\frac{\mathfrak{c}_p}{a})^{\frac1a}.
		$$
		The equality holds if and only if $M$ is isometric to $\mathbb{R}^3\setminus B_{r_{0}}$.
		
		\qed

		
		\appendix
		
		\section{Solve the ODE equations}
		In this Appendix, we solve the ordinary differential equations in $[r_0, +\infty)$:
		\begin{equation*}
			\begin{cases}
				\alpha'(t) - (2 a + 1) \eta(t) f'(t) \alpha(t) - a f'(t) \beta(t) = 0, \\
				\beta'(t) + (2 a + 1) (\eta(t))^2 f'(t) \alpha(t) = 0, \\
				\gamma'(t) + f'(t) \alpha(t) = 0,
			\end{cases}
		\end{equation*}
		where
		\begin{eqnarray*}
			f (t)&=& 1 - \int_{t}^{\infty} \mathfrak{c}_{p} s^{-a-1}\left(1+\frac{m}{2s}\right)^{-2a}ds, \\
			\eta(t)&=& \mathfrak{c}_{p}^{-1}t^{a}\left(1 + \frac{m}{2 t}\right)^{2a-1}\left(1- \frac{m}{2t}\right),\\
		\end{eqnarray*}
		and $m$ and $\mathfrak{c}_{p}$	are two constants satisfying $m\neq0$, and $|m| = 2 \left(I_{a}(\frac{m}{2r_0})\mathfrak{c}_{p}\right)^{\frac{1}{a}}.$

		\begin{Remark} \label{Rem:Beta}
			When $r_{0}=\frac{m}{2}$, $I_{a}(\frac{m}{2r_0})=I_{a}(1)$ is indeed  half of the Beta function $\mathcal{B}(a,a)$. Recall that
			\begin{equation} \label{equ:Beta}
				\mathcal{B}(p,q) = \int_{0}^{1} t^{p-1}(1-t)^{q-1} dt.
			\end{equation}
			If we take $t = \frac{1}{1+x}$, then
			$$
			\mathcal{B}(p,q) = \int_{0}^{+ \infty} x^{q-1}(1+x)^{-(p+q)} dx.
			$$
			While for $I_{a}(1) = \int_{0}^{1} s^{a-1}(s+1)^{-2a} ds$, if we take
			$t = \frac{1}{s}$, then $I = \int_{1}^{+\infty} t^{a-1}(1+t)^{-2a} dt$,
			so
			$$
			I_{a}(1) = \frac{1}{2} \int_{0}^{+\infty} s^{a-1}(s+1)^{-2a} ds
			= \frac{1}{2} \mathcal{B}(a,a).
			$$
		\end{Remark}

		\begin{Proposition}\label{ode-solution}
			The solution to \eqref{equ: differential equations} is given by the following:
			\begin{equation*}
				\begin{aligned}
					\alpha(t)
					=& t\left(1+\frac{m}{2t}\right)^{2}
					\left\{ \left( C_{2}\mathfrak{c}_{p}+C_{1}\frac{a}{m}(1-f(t)) \right)
					\eta(t)-C_{1}\frac{1}{m} \right\},
				\end{aligned}
			\end{equation*}
			\begin{equation*}
				\begin{aligned}
					\beta(t)
					=&- \eta(t)\alpha(t)
					+ \left( C_{2}\mathfrak{c}_{p}\frac{m}{a}+C_{1}(1-f(t)) \right) \mathfrak{c}_{p}^{-2} t^{2a}\left( 1+\frac{m}{2t}\right)^{4a},
				\end{aligned}
			\end{equation*}
			\begin{equation*}
				\begin{aligned}
					\gamma(t)
					=&- \mathfrak{c}_{p}^{2} t^{-2a}\left( 1+\frac{m}{2t}\right)^{-4a} \eta(t) \alpha(t)
					-  \left( C_{2}\mathfrak{c}_{p}\frac{m}{a}+C_{1}(1-f(t)) \right)+C_3,
				\end{aligned}
			\end{equation*}
			where $C_1, C_2, C_3\in \mathbb{R}$.
			If $r_{0}=\frac{m}{2}$, $m>0$, then
			$$
			\begin{cases}
				\eta(r_{0})=0, \\
				\alpha(r_{0})=-2C_{1}, \\
				\beta(r_{0})= \left( C_{2}\mathfrak{c}_{p}\frac{m}{a}+C_{1} \right)  2^{4a}(I_{a}(1))^{2}, \\
				\gamma(r_{0})= -\left( C_{2}\mathfrak{c}_{p}\frac{m}{a}+C_{1} \right)+C_3.
			\end{cases}
			$$
		\end{Proposition}
		
		\begin{proof}From the first two equations, we get an ordinary differential equation of second order:
			\begin{equation} \label{equ:ODEofAlpha}
				t^{2}\left(1 + \frac{m}{2t}\right)^{2} \alpha''(t) + t\left(1 + \frac{m}{2t}\right)\left( - a + (a+2)\frac{m}{2t} \right) \alpha'(t) - 2(2 a + 1) \frac{m}{2t}  \alpha(t) =0
			\end{equation}
			Assume $$\alpha(t) = t^{a - 1}\left(1+ \frac{m}{2 t}\right)^{2 a} \bar{\alpha}(t).$$ Then $\bar{\alpha}(t)$ satisfies
			\begin{equation} \label{equ:ODEofAlpha-1}
				0 = t\left(t + \frac{m}{2}\right)^{2} \bar{\alpha}''(t) + \left(t + \frac{m}{2}\right)\left( (a - 2) t - a \frac{m}{2} \right) \bar{\alpha}'(t) + 2(1 - a) t \bar{\alpha}(t).
			\end{equation}
			We observe that $\bar{\alpha}(t) = t^{2} - (\frac{m}{2})^2$ is a special quadratic polynomial solution to \eqref{equ:ODEofAlpha-1},
			It follows that $$\alpha(t)=  t^{a + 1}\left(1+ \frac{m}{2 t}\right)^{2 a+1}\left(1 - \frac{m}{2t}\right)$$ is a special solution to  \eqref{equ:ODEofAlpha}.
			
			On the other hand, assume $$\alpha(t) = t^{a + 1}\left(1 + \frac{m}{2 t}\right)^{2 a + 1}\left(1 - \frac{m}{2t}\right) \tilde{\alpha}(t).$$ Then $\tilde{\alpha}(t)$ satisfies
			\begin{equation} \label{equ: tilde alpha differential equation}
				t \left(t + \frac{m}{2}\right) \left(t - \frac{m}{2}\right) \tilde{\alpha}''(t) + \left( a(\frac{m}{2})^2 + 2(1 - a)\frac{m}{2} t + (a + 2)t^2 \right) \tilde{\alpha}'(t) = 0.
			\end{equation}
			We obtain
			$$
			\tilde{\alpha}'(t) = C_{1}t^{-a-2} \left(1 + \frac{m}{2 t}\right)^{-2 a}\left(1 - \frac{m}{2t}\right)^{-2},
			$$
			where $C_{1}\in\mathbb{R}$ is a constant.
			Since $\tilde{\alpha}'(t)\ge 0$ and $\tilde{\alpha}'(t) = O(t^{-a-2}) (t\rightarrow+\infty)$, we see that $$C_2:=\lim\limits_{t\rightarrow+\infty}\tilde{\alpha}(t)<+\infty.$$
			We get that
			\begin{eqnarray*}
				\tilde{\alpha}(t)
				&=& C_{2}-C_{1}\int_{t}^{+\infty}s^{-a-2}\left(1+\frac{m}{2s}\right)^{-2a}\left(1-\frac{m}{2s}\right)^{-2}ds.
			\end{eqnarray*}
			Using integration by parts, we see
			$$
			\begin{aligned}
				\tilde{\alpha}(t)
				=& C_{2}
				+C_{1}\frac{1}{m} s^{-a}\left(1+\frac{m}{2s}\right)^{-2a+1}\left(1-\frac{m}{2s}\right)^{-1} \Big|_{t}^{+\infty}
				+C_{1}\frac{a}{m} \int_{t}^{+\infty}s^{-a-1}\left(1+\frac{m}{2s}\right)^{-2a}ds \\
				=& C_{2}
				-C_{1}\frac{1}{m} t^{-a}\left(1+\frac{m}{2t}\right)^{-2a+1}\left(1-\frac{m}{2t}\right)^{-1}
				+C_{1}\frac{a}{m} (\frac{m}{2})^{-a}I_{a}(\frac{m}{2t}),
			\end{aligned}
			$$
			Therefore, the  solution to \eqref{equ:ODEofAlpha} is given by
			\begin{equation*}
				\begin{aligned}
					\alpha(t)
					=& \left\{ C_{2}+C_{1}\frac{a}{m} (\frac{m}{2})^{-a}I_{a}(\frac{m}{2t}) \right\}
					t^{a + 1}(1 + \frac{m}{2 t})^{2a+1}(1 - \frac{m}{2t})
					-C_{1}\frac{1}{m}t(1+\frac{m}{2t})^{2} \\
					=& t\left(1+\frac{m}{2t}\right)^{2}
					\left\{ \left( C_{2}\mathfrak{c}_{p}+C_{1}\frac{a}{m} \frac{I_{a}(\frac{m}{2t})}{I_{a}(\frac{m}{2r_0})}  \right)
					\eta(t)-C_{1}\frac{1}{m} \right\}.
				\end{aligned}
			\end{equation*}
			In the second equality, we have used the definition of $\eta(t)$.
			
			To calculate $\beta$, we need the derivative of $\alpha$. By computation, we get
			\begin{equation*}
				\begin{aligned}
					\alpha'(t)
					=& (a+1) (1+\frac{m}{2t})(1-\frac{m}{2t})
					\left\{ \left( C_{2}\mathfrak{c}_{p}+C_{1}\frac{a}{m}\frac{I_{a}(\frac{m}{2t})}{I_{a}(\frac{m}{2r_0}}  \right)
					\eta(t)-C_{1}\frac{1}{m} \right\} \\
					&+ m
					\left( C_{2}\mathfrak{c}_{p}+C_{1}\frac{a}{m} \frac{I_{a}(\frac{m}{2t})}{I_{a}(\frac{m}{2r_0})}\right)
					\mathfrak{c}_{p}^{-1}t^{a-1}(1 + \frac{m}{2 t})^{2a} \\
					=& (a+1) \eta(t)f'(t) \alpha(t)
					+ \left( C_{2}\mathfrak{c}_{p}+C_{1}\frac{a}{m} \frac{I_{a}(\frac{m}{2t})}{I_{a}(\frac{m}{2r_0})}\right) m
					\mathfrak{c}_{p}^{-2}t^{2a}(1 + \frac{m}{2 t})^{4a}f'(t) \\
				\end{aligned}
			\end{equation*}
			It follows from the ODE \eqref{equ: differential equations} that
			\begin{equation*}
				\begin{aligned}
					\beta(t)
					=&\frac{1}{a f'(t)}\alpha'(t) -\frac{2 a + 1}{a} \eta(t) \alpha(t)\\
					=& - \eta(t)\alpha(t)
					+ \left( C_{2}\mathfrak{c}_{p}\frac{m}{a}+C_{1} \frac{I_{a}(\frac{m}{2t})}{I_{a}(\frac{m}{2r_0})} \right) \mathfrak{c}_{p}^{-2} t^{2a}\left( 1+\frac{m}{2t}\right)^{4a}.
				\end{aligned}
			\end{equation*}
			Finally, using the ODE \eqref{equ: differential equations} again, we notice that
			\begin{equation*}
				\begin{aligned}
					&\frac{d}{dt} \left( - \mathfrak{c}_{p}^{2} t^{-2a}\left(1+\frac{m}{2t}\right)^{-4a} (\beta(t)+2\eta(t)\alpha(t)) \right)  \\
					=&- \mathfrak{c}_{p}^{2} t^{-2a}\left(1+\frac{m}{2t}\right)^{-4a}  \left\{ \beta'(t)+2\eta'(t)\alpha(t)+2\eta(t)\alpha'(t)
					-2a\eta(t)f'(t)(\beta(t)+2\eta(t)\alpha(t)) \right\}\\
					=&- f'(t) \alpha(t).\\
				\end{aligned}
			\end{equation*}
			It follows that
			\begin{equation*}
				\begin{aligned}
					\gamma(t)
					=&- \mathfrak{c}_{p}^{2} t^{-2a}\left(1+\frac{m}{2t}\right)^{-4a} (\beta(t)+2\eta(t)\alpha(t))+C_{3}\\
					=&- \mathfrak{c}_{p}^{2} t^{-2a}\left(1+\frac{m}{2t}\right)^{-4a} \eta(t) \alpha(t)
					- \left( C_{2}\mathfrak{c}_{p}\frac{m}{a}+C_{1} \frac{I_{a}(\frac{m}{2t})}{I_{a}(\frac{m}{2r_0})} \right) +C_3,
				\end{aligned}
			\end{equation*}
			where $C_3\in\mathbb{R}$ is a constant.
		\end{proof}
		
		\begin{Remark}
			The functions $\alpha(t)$, $\beta(t)$, $\gamma(t)$ are still the solutions of the ordinary differential equations when $m<0$. 
		\end{Remark}
		
		{
			\begin{Proposition}\label{alpha>0}
				$\alpha(t)\ge 0$ on $[r_{0},+\infty)$ if and only if
				$$
				C_{2} \ge 0, \hbox{ and }
				\left( C_{2}\mathfrak{c}_{p}+C_{1}\frac{a}{m} \right)
				\eta(r_{0}) \ge C_{1}\frac{1}{m}. 			$$
				In particular, if $m>0$, $r_0=\frac{m}{2}$, then
				$\alpha(t)\ge 0$ on $[r_{0},+\infty)$ if and only if $C_2\ge 0$ and $C_1\le 0$.
			\end{Proposition}
		}
		
		
		\begin{proof}
			First we consider the necessary conditions of $\alpha(t) \ge 0, \forall t \in [r_{0},+\infty)$.
			
			$(i).$ $\alpha(r_{0})\ge 0$, i.e. $\left( C_{2}\mathfrak{c}_{p}+C_{1}\frac{a}{m} \right)
			\eta(r_{0}) \ge C_{1}\frac{1}{m}$.
			
			$(ii).$ $\lim\limits_{t\to\infty} \alpha(t)\ge 0$. Notice that $t(1+\frac{m}{2t})^{2}\eta(t)>0$ as if $t>r_{0}$.
			Then $\alpha(t)$ and $\frac{\alpha(t)}{t(1+\frac{m}{2t})^{2}\eta(t)}= \left( C_{2}\mathfrak{c}_{p}-C_{1}\xi(t) \right)$ has the same sign on $(r_{0},+\infty)$. Here
			$$
			\xi(t)=\mathfrak{c}_{p} \int_{t}^{+\infty}s^{-a-2}\left(1+\frac{m}{2s}\right)^{-2a}\left(1-\frac{m}{2s}\right)^{-2}ds.
			$$
			Recall \eqref{equ:LimitofXi},
			$\lim\limits_{t\to\infty} \left( C_{2}\mathfrak{c}_{p}-C_{1}\xi(t) \right) = C_{2}\mathfrak{c}_{p}$.
			i.e. $\lim\limits_{t\to\infty} \alpha(t)\ge 0$ is equivalent to $C_{2} \ge 0$.
			
			Now we prove that the conditions are sufficient. \\
			$(1).$ When $r_{0}>\frac{|m|}{2}$, $\eta(t)>0$ on $[r_{0},+\infty)$,
			$$
			\xi(t)=\frac{C_{1}\frac{a}{m} \frac{I_{a}(\frac{m}{2t})}{I_{a}(\frac{m}{2r_0})} \eta(t) - C_{1}\frac{1}{m}}{\eta(t)}
			$$
			
			It's easy to see that $C_{2}\mathfrak{c}_{p}-C_{1}\xi(t)$ is monotone on $[r_{0},+\infty)$, and by the conditions, $C_{2}\mathfrak{c}_{p}-C_{1}\xi(t)\ge0$ when $t=r_{0}$ and when $t\rightarrow+\infty$. Thus when the conditions hold, we have that $\alpha(t)\ge0$.
			
			\noindent $(2).$ When $r_{0}=\frac{m}{2}$, the conditions become $C_{2}\ge0,C_{1}\le0$.
			Notice that
			$$
			\begin{aligned}
				& \frac{d}{dt} \frac{\alpha(t)}{t^{a+1}(1+\frac{m}{2t})^{2a+1}} \\
				=& \frac{d}{dt}
				\left( C_{2}(1-\frac{m}{2t})-C_{1}\frac{1}{m}t^{-a}(1+\frac{m}{2t})^{-2a+1}+C_{1}\frac{a}{m}(1-\frac{m}{2t})\int_{t}^{+\infty}s^{-a-1}(1+\frac{m}{2t})^{-2a}ds \right) \\
				=& \frac{m}{2t^{2}} \mathfrak{c}_{p}^{-1}
				\left( C_{2}\mathfrak{c}_{p}+C_{1}\frac{a}{m}\frac{I_{a}(\frac{m}{2t})}{I_{a}(\frac{m}{2r_0})}+C_{1}\frac{1}{m}t^{-a-2}(1+\frac{m}{2t})^{-2a} \right) .
			\end{aligned}	
			$$
			And since that
			$$
			\begin{aligned}
				\frac{d}{dt} \left( C_{2}\mathfrak{c}_{p}+C_{1}\frac{a}{m}\frac{I_{a}(\frac{m}{2t})}{I_{a}(\frac{m}{2r_0})}+C_{1}\frac{1}{m}\mathfrak{c}_{p}t^{-a}(1+\frac{m}{2t})^{-2a} \right)
				= -2C_{1}\frac{a}{m}\mathfrak{c}_{p}t^{-a-1}(1+\frac{m}{2t})^{-2a-1} \ge 0,
			\end{aligned}	
			$$
			then when $t$ increases from $r_{0}$ to $+\infty$, the value of
			$$
			\left( C_{2}\mathfrak{c}_{p}+C_{1}\frac{a}{m}\frac{I_{a}(\frac{m}{2t})}{I_{a}(\frac{m}{2r_0})}+C_{1}\frac{1}{m}\mathfrak{c}_{p}t^{-a}(1+\frac{m}{2t})^{-2a} \right)
			$$
			changes from $ \left( C_{2}\mathfrak{c}_{p}+C_{1}\frac{a}{m}+C_{1}\frac{1}{m}(I_{a}(1))^{-1}2^{-2a} \right) $ to $C_{2}\mathfrak{c}_{p}$.
			Recall that
			$$
			\alpha'(t) = (a+1) \eta(t)f'(t) \alpha(t)
			+ \left( C_{2}\mathfrak{c}_{p}+C_{1}\frac{a}{m} \frac{I_{a}(\frac{m}{2t})}{I_{a}(\frac{m}{2r_0})}\right) m
			\mathfrak{c}_{p}^{-2}t^{2a}(1 + \frac{m}{2 t})^{4a}f'(t),
			$$
			now we talk about the sign of $\alpha$ in two cases. \\
			\noindent{\bf Case 1.} If $C_{2}\mathfrak{c}_{p}+C_{1}\frac{a}{m}\ge0$, then $\alpha'(t)\ge0$, $\alpha(t)\ge\alpha(r_{0})=0$. \\
			\noindent{\bf Case 2.}  If $C_{2}\mathfrak{c}_{p}+C_{1}\frac{a}{m}<0$, then the derivative of
			$\frac{\alpha(t)}{t^{a+1}(1+\frac{m}{2t})^{2a+1}}$ is increasing from a negative number to a non-negative number when t changes from $r_{0}$ to $+\infty$. Thus the value of $\frac{\alpha(t)}{t^{a+1}(1+\frac{m}{2t})^{2a+1}}$ is monotone non-increasing from $r_{0}$ to $r_{1}$ and is monotone non-decreasing from $r_{1}$ to $+\infty$, where $r_{1}$ is the number such that
			$$
			\left( C_{2}\mathfrak{c}_{p}+C_{1}\frac{a}{m}\frac{I_{a}(\frac{m}{2r_{1}})}{I_{a}(\frac{m}{2r_0})}+C_{1}\frac{1}{m}\mathfrak{c}_{p}r_{1}^{-a}(1+\frac{m}{2r_{1}})^{-2a} \right) =0.
			$$
			While $t=r_{1}$,
			$$
			\begin{aligned}
				\alpha(r_{1})=& r_{1}\left(1+\frac{m}{2r_{1}}\right)^{2}
				\left\{ \left( C_{2}\mathfrak{c}_{p}+C_{1}\frac{a}{m} \frac{I_{a}(\frac{m}{2r_{1}})}{I_{a}(\frac{m}{2r_0})}  \right)
				\eta(r_{1})-C_{1}\frac{1}{m} \right\} \\
				=& r_{1}\left(1+\frac{m}{2r_{1}}\right)^{2}
				\left\{ -C_{1}\frac{1}{m}\mathfrak{c}_{p}r_{1}^{-a}(1+\frac{m}{2r_{1}})^{-2a}
				\eta(r_{1})-C_{1}\frac{1}{m} \right\}
				\ge 0.
			\end{aligned}
			$$
			Thus $$\frac{\alpha(t)}{t^{a+1}(1+\frac{m}{2t})^{2a+1}}\ge0$$ holds on $[r_{0},+\infty)$, which implies that $\alpha(t)\ge0$ on $[r_{0},+\infty)$.
		\end{proof}
		
		
	
	\
	
	

\end{document}